\documentclass[11pt]{article}
\usepackage{latexsym,amsmath,amscd,amssymb,framed,graphics,mathrsfs}
\usepackage{enumerate}
\usepackage{graphicx}
\usepackage{subcaption}
\usepackage{color}
\usepackage{framed}
\usepackage{comment}
\usepackage{booktabs}
\usepackage{nicematrix}

\usepackage{fancyhdr}
\usepackage{tikz}
    \usetikzlibrary{matrix}
    \usetikzlibrary{decorations.markings}
\usepackage[authoryear,sort&compress]{natbib}
\bibpunct{[}{]}{;}{n}{,}{,}
\usepackage[colorlinks]{hyperref}
\usepackage{url}
\usepackage[all]{xy}
\usepackage{url}

\newcommand{\rem}[1]{}

\newcommand\curl{\mathord{\operatorname{curl}\,}}

\makeatletter

\newenvironment{proof}[1][Proof]{\noindent\textbf{#1.} }{\ \rule{0.5em}{0.5em}}

\def\XXint#1#2#3{{\setbox0=\hbox{$#1{#2#3}{\int}$ }
\vcenter{\hbox{$#2#3$ }}\kern-.5\wd0}}

\definecolor{bgd}{RGB}{153,0,51}      

\begin{document}

\newtheorem{theorem}{Theorem}[section]
\newtheorem{definition}[theorem]{Definition}
\newtheorem{lemma}[theorem]{Lemma}
\newtheorem{remark}[theorem]{Remark}
\newtheorem{proposition}[theorem]{Proposition}
\newtheorem{corollary}[theorem]{Corollary}
\newtheorem{example}[theorem]{Example}



\title{Infinite-dimensional Lagrange--Dirac systems with\\ boundary energy flow I: Foundations}

\fancyhead[c]{Infinite-dimensional Lagrange--Dirac systems I}
\fancyfoot[c]{\thepage}

\author{Fran\c{c}ois Gay--Balmaz$^{1}$, \'{A}lvaro Rodr\'{i}guez Abella$^{2}$ and Hiroaki Yoshimura$^{3}$}

\addtocounter{footnote}{1}
\footnotetext{Division of Mathematical Sciences, Nanyang Technological University, 21 Nanyang Link, Singapore 637371.
\texttt{francois.gb@ntu.edu.sg}
\addtocounter{footnote}{1}}

\footnotetext{Department of Applied Mathematics, ICAI School of Engineering, Comillas Pontifical University, Madrid, Spain.
\texttt{arabella@comillas.edu}
\addtocounter{footnote}{1}}

\footnotetext{School of Science and Engineering, Waseda University. 3--4--1, Okubo, Shinjuku, Tokyo 169-8555, Japan.
\texttt{yoshimura@waseda.jp}
\addtocounter{footnote}{1}}

\date{ }
\maketitle
\makeatother



\begin{abstract}
A new geometric approach to systems with boundary energy flow is developed using infinite-dimensional Dirac structures within the Lagrangian formalism. This framework satisfies a list of consistency criteria with the geometric setting of finite-dimensional mechanics. In particular, the infinite-dimensional Dirac structure can be constructed from the canonical symplectic form on the system's phase space; the system’s evolution equations can be derived equivalently from either a variational
perspective or a Dirac structure perspective; the variational principle employed is a direct extension of Hamilton’s principle in classical mechanics; and the approach allows for a process of system interconnection within its formulation. This is achieved by developing an appropriate infinite dimensional version of the previously developed Lagrange--Dirac dynamical systems.
A key step in this construction is the careful choice of a suitable dual space to the configuration space—specifically, a subspace of the topological dual that captures the system's behavior in both the interior and the boundary, while allowing for a natural extension of the canonical geometric structures of mechanics.
This paper focuses on systems where the configuration space consists of differential forms on a smooth manifold with a boundary. To illustrate our theory, several examples, including nonlinear wave equations, the telegraph equation, and the Maxwell equations are presented.
\end{abstract}


\maketitle

\section{Introduction}

Variational and geometric structures underlying infinite-dimensional dynamical systems play a crucial role in the modeling and structure-preserving discretization of these systems. This becomes particularly important when such systems interact with their surroundings through their boundaries. A predominant approach to addressing this type of problem has been developed from the Hamiltonian perspective, specifically within the port-Hamiltonian framework \cite{VdSMa2002,MaVdS2005,RaCaVdSSt2020}. The aim of the present paper is to establish a geometric framework for systems with boundary energy flow that meets the following criteria:
\begin{itemize}
\item[\rm (i)] The system's evolution equations can be derived equivalently from either a variational perspective or a Dirac structure perspective.
\item[\rm (ii)] The infinite-dimensional Dirac structures involved directly extend the canonical symplectic structure 
$ dq \wedge dp$ of classical mechanics.
\item[\rm (iii)] The variational principle employed is a direct extension of Hamilton's principle $\delta \!\int\! L(q, \dot q) dt$ $=0$ from classical mechanics.
\item[\rm (iv)] The geometric setting does not impose a specific form or regularity on the involved Lagrangian density.
\item[\rm (v)] The approach allows for the systematic development of the process of system interconnection within its formulation.
\item[\rm (vi)] In presence of a Lie group symmetry, the approach allows for a reduction process within its geometric framework.
\end{itemize}

To develop an approach satisfying (i)--(vi), we will extend the Lagrange--Dirac dynamical system from \cite{YoMa2006a} to the infinite-dimensional setting. This framework is based on the use of Dirac structures \cite{Co1990} in conjunction with the variational Lagrangian formalism of mechanics. A key advantage of employing the Lagrange--Dirac approach to dynamical systems is that it provides a unified geometric formulation for systems that can be degenerate and nonholonomic, while also admitting an associated variational formulation \cite{YoMa2006b}. Additionally, this approach allows for an interconnection process from both the variational and Dirac perspectives \cite{JaYo2014}. Reduction by symmetries within the Lagrange--Dirac setting was developed in \cite{YoMa2007,YoMa2009,GBYo2015}. Some extension to classical field theories was made in \cite{VaYoLe2012} by proposing the notion of multi-Dirac structures, leading to the Lagrange--Dirac field equations. More recently the Lagrange--Dirac setting was extended to thermodynamic systems \cite{GBYo2018,GBYo2020}. An infinite-dimensional version of the Lagrange--Dirac dynamical systems was first proposed by \cite{GBYo2015} in the context of Lie--Dirac reduction over semi-direct products, applied to the  incompressible ideal fluids and compressible magnetohydrodynamics, as well as to second-order Rivlin--Ericksen fluids seen as an infinite-dimensional nonholonomic system.
However, this approach did not address the treatment of boundary energy flow

\medskip

The aim of this paper is to develop the foundations for a geometric and variational framework for infinite-dimensional systems with boundary energy flow that satisfy conditions (i)--(vi). While the theory will ultimately be developed for a broad class of infinite-dimensional configuration manifolds of maps and Lagrangian functions, this paper focuses on the foundational aspects for the following specific situation.

(i) \textit{Configuration manifolds.}
We restrict to infinite-dimensional vector spaces and focus on the space $V=C^\infty(\mathcal B)$ of smooth functions on a bounded domain with smooth boundary, and, more generally, on the space of smooth $k$-forms on a smooth manifold with boundary, $V=\Omega^k(M)$.

(ii) \textit{Lagrangian functions.}
We consider Lagrangians $L:TV\to\mathbb R$, where $TV=V \times V$, defined through a density, that is, for the case $V=C^\infty(\mathcal B)$, we have
\begin{equation*}
L(\varphi,\nu)=\int_{\mathcal B}\mathfrak L(\varphi(x),\nu(x),\nabla\varphi(x))\,dx,\qquad(\varphi,\nu)\in TC^\infty(\mathcal B)=C^\infty(\mathcal B)\times C^\infty(\mathcal B),
\end{equation*}
where $\mathfrak L:\mathbb R\times\mathbb R\times\mathbb R^m\to\mathbb R$; and for the case $V=\Omega^k(M)$, we have
\begin{equation*}
L(\varphi,\nu)=\int_M\mathscr L(\varphi,\nu,{\rm d}\varphi),\qquad(\varphi,\nu)\in T\Omega^k(M)=\Omega^k(M)\times\Omega^k(M),
\end{equation*}
where $\mathscr L:\textstyle\bigwedge^k T^* M\times_M\bigwedge^kT^* M\times_M\bigwedge^{k+1} T^* M\to\bigwedge^m T^* M$. For these classes of Lagrangians, their partial derivatives are seen to lie in a vector subspace of the topological dual space, referred to as the \emph{restricted dual}, which maintains a nondegenerate correspondence with the original space.

\medskip
In this paper, we will concentrate on the properties (i)--(iv) above, while (v) and (vi) will be the subject of future works. Specifically, we will demonstrate that property (i) holds precisely as it does in the finite-dimensional Lagrange--Dirac framework, now extended to accommodate boundary energy flow. Regarding (ii) we will highlight how our approach consistently extends the geometric structures of finite-dimensional mechanics, by showing that it is based on an infinite-dimensional version of the canonical symplectic form, from which a Dirac structure can be associated following the classical definition. The key step for this is the consideration of an appropriate dual space to the configuration space, which incorporates functions or differential forms on the boundary. This subspace of the topological dual effectively accounts for both the boundary effects in the system and the boundary terms in the derivative of the Lagrangian. Regarding (iii) we will show that the solutions of the resulting Lagrange--Dirac system can be characterized as the critical curves of a variational formulation that consistently extends the Hamilton principle of classical mechanics, namely, the Lagrange--d'Alembert--Pontryagin principle which incorporates the configuration variable as well as the velocity and momentum variables in its formulation. Property (iv) will be illustrated using the example of the nonlinear wave equation, demonstrating that while our approach can inherently handle arbitrary Lagrangians, achieving such an extension would be more challenging with previous methods.

\medskip

The analogies between the finite and infinite-dimensional setting are illustrated in Table \ref{comparison}.

\medskip

In part II of this paper \cite{GBRAYo2025II}, we extend these results to systems described by bundle-valued $k$-forms, with application to gauge and particle field theories. For future work, we plan to consider systems on infinite-dimensional manifolds of maps. Such a theory would encompass fluid dynamics in a fixed domain, where the configuration manifold is the group of diffeomorphisms of $\mathcal B$, $\operatorname{Diff}(\mathcal B)$, as well as continuum mechanics with moving boundary, where the configuration manifold is the manifold of embeddings of $\mathcal B$ on $\mathbb R^m$, $\operatorname{Emb}(\mathcal B,\mathbb R^m)$. Additionally, reduction by symmetries could be achieved using the relabeling symmetry and material frame indifference, yielding the so-called spatial and convective representations, \cite{GaMaRa2012}.
Finally, we aim to develop the interconnection of systems within this geometric framework, see \cite{GBYo2024} for preliminary results.

\begin{table}[ht!]
\centering
\begin{NiceTabular}{p{7.5cm} | p{7.5cm}}
\toprule
\textbf{Finite dimensional} & \textbf{Infinite-dimensional}\\
\midrule\midrule
\multicolumn{2}{c}{\textsf{Configuration manifold \& velocity and momentum phase spaces}}\\
\midrule 
$\begin{array}{l}
Q \;\;\text{manifold} \\
TQ\ni(q,v) \\
T^*Q \ni(q,p)
\end{array}$
&
$\begin{array}{l}
V= C^\infty( \mathcal{B})\;\; \text{Fr\'echet space}\\
V^\star = C^\infty( \mathcal{B} ) \times C^\infty( \partial \mathcal{B})\;\; \text{restricted dual}\\
TV=V \times V\ni( \varphi, \nu )\\
T^\star V= V \times V^\star\ni( \varphi , \alpha , \alpha _ \partial )
\end{array}$
\\
\midrule
\multicolumn{2}{c}{\textsf{Canonical 1-form \& canonical symplectic form}}\\
\midrule
$\Theta (q,p) \cdot ( \delta q, \delta p)= \left\langle p, \delta q \right\rangle$ &
$\begin{array}{l}
\Theta ( \varphi , \alpha , \alpha _ \partial ) \cdot ( \delta \varphi , \delta \alpha , \delta \alpha _ \partial ) = \left\langle (\alpha , \alpha _ \partial ), \delta \varphi \right\rangle \\
\quad=\displaystyle \vspace{0.2cm}\int_ \mathcal{B} \alpha \delta \varphi\,dx + \int_{ \partial \mathcal{B} } \alpha _ \partial \delta \varphi\,ds
\end{array}$
\\[0.25ex]
$ \Omega = - {\rm d} \Theta $ & $ \Omega = - {\rm d} \Theta $ \\
\midrule
\multicolumn{2}{c}{\textsf{Dirac structure}}\\
\midrule
$D= \operatorname{graph} \Omega ^\flat $
&
$D= \operatorname{graph} \Omega ^\flat$
\\
$\subset T(T^*Q) \oplus T^*(T^*Q)$ & $\subset T(T^\star V) \oplus T^\star(T^\star V) $
\\ 
\midrule
\multicolumn{2}{c}{\textsf{Lagrangian}}\\
\midrule
$L:TQ \rightarrow \mathbb{R}$ 
& $L:TV \rightarrow \mathbb{R}$\\
& $\displaystyle\vspace{0.2cm} L( \varphi , \nu )= \int_ \mathcal{B} \mathfrak{L}( \varphi , \nu , \nabla \varphi )\,dx$
\\
\midrule
\multicolumn{2}{c}{\textsf{Force}}\\
\midrule
$F:TQ \rightarrow T^*Q$ 
& $F:TV \rightarrow T^\star V$ \\
&$\displaystyle\vspace{0.2cm} F( \varphi , \nu )= ( \varphi , \mathcal{F} ( \varphi , \nu ), \mathcal{F} _ \partial ( \varphi, \nu ))$
\\ 
\midrule
\multicolumn{2}{c}{\textsf{Lagrange--Dirac system}}\\ 
\midrule
$
\begin{array}{l}
\big((q,p, \dot q,\dot p), {\mathrm d}_D L(q,v)-\widetilde{F}(q,v)\big)\\
\quad\in D_{T^\star V} (q,p)
\end{array}$
&
$
\!\!\!\!\begin{array}{l}
\big((\varphi,\alpha,\alpha_\partial, \dot\varphi,\dot\alpha,\dot\alpha_\partial), {\mathrm d}_D L(\varphi,\nu)-\widetilde{F}(\varphi,\nu)\big)\\
\quad\in D_{T^\star V} (\varphi,\alpha,\alpha_\partial)
\end{array}\!\!\!\!$
\\
\vspace{-0.2cm} & \vspace{-0.2cm}\\
for $(q(t),v(t),p(t)) \in TQ \oplus T^*Q$ & for $(\varphi (t), \nu (t), \alpha (t), \alpha _ \partial (t)) \in TV \oplus T^\star V$
\\ 
\midrule
\multicolumn{2}{c}{\textsf{Variational principle}}\\
\midrule
$\begin{array}{c}
\displaystyle\vspace{0cm}\delta\int_{t_{0}}^{t_{1}}L(\varphi,v)+ \left< p, \dot q-v\right>dt\\
\displaystyle+\int_{t_{0}}^{t_{1}} \left<F(q,\dot{q}),\delta{q}\right> dt=0
\end{array}$ 
& $\begin{array}{c}
\displaystyle\vspace{0cm}\delta\int_{t_{0}}^{t_{1}}L(\varphi,\nu)+ \left< (\alpha, \alpha_{\partial}), \dot\varphi-\nu\right>dt\\
\displaystyle+\int_{t_{0}}^{t_{1}} \left<F(\varphi,\dot{\varphi}),\delta{\varphi}\right> dt=0
\end{array}$\\
\vspace{-0.2cm} & \vspace{-0.2cm}\\
for $(q(t),v(t),p(t)) \in TQ \oplus T^*Q$ & for $(\varphi (t), \nu (t), \alpha (t), \alpha _ \partial (t)) \in TV \oplus T^\star V$\\
\bottomrule
\end{NiceTabular}
\caption{Schematic correspondence between the finite and infinite-dimensional Lagrange--Dirac settings.}
\label{comparison}
\end{table}

\paragraph{Plan of the paper}
In \S \ref{sec:prelim}, we briefly review the theory of finite-dimensional Lagrange--Dirac dynamical systems and recall some facts on Fréchet spaces and their duals. The main ideas of the paper are introduced in \S \ref{sec:functions}, where the theory of infinite-dimensional Lagrange--Dirac dynamical systems with boundary energy flow on the space of smooth functions on a bounded domain with smooth boundary is presented. In particular, the restricted dual is defined, which leads to the restricted iterated bundles and the restricted Tulczyjew triple. Furthermore, we show that there is a variational principle associated with the infinite-dimensional Lagrange--Dirac dynamical system, given by a Lagrange--d'Alembert--Pontryagin variational principle, thereby  consistently extending the situation of finite-dimensional mechanics reviewed in \S\ref{sec:prelim}, see also \cite{RAGBYo2023}. Examples are then provided to illustrate our theory for Lagrange–Dirac systems with body and boundary forces, including a vibrating membrane, nonlinear wave equations, and an ideal one-dimensional transmission line. In \S \ref{sec:kforms} an extension is made to the case of Lagrange--Dirac dynamical systems  with boundary energy flow on the space of $k$-forms on a smooth manifold with boundary.
Some remarks are also made to compare our proposal with the Stokes--Dirac structures approach introduced in \cite{VdSMa2002}. Then, the example of electromagnetism illustrates the theory of Lagrange--Dirac systems on the space of $k$-forms.

\section{Preliminaries}\label{sec:prelim}

Here we recall the essential notions about Dirac structures and Lagrange--Dirac mechanical systems in the finite dimensional setting. An extended development of the definitions and results introduced here can be found in \cite{Co1990,YoMa2006a,YoMa2006b}. We also recall some elementary notions about Fréchet spaces and dual systems, as can be found in \cite{Ru1991,MeMeVoRa1997}.
\medskip

\subsection{Dirac structures on manifolds}\label{sec:diracprelim}

Let $M$ be a finite dimensional smooth manifold. A fibered interior product may be defined canonically on its Pontryagin bundle $TM \oplus T^*M \rightarrow M$ as follows, 
\begin{equation*}
\left\langle \!\left\langle \cdot,\cdot \right\rangle \!\right\rangle :(TM\oplus T^*M)\times_M(TM\oplus T^*M)\to\mathbb R,\qquad\left\langle \!\left\langle(v_1,\alpha_1),(v_2,\alpha_2)\right\rangle \!\right\rangle=\alpha_1(v_2)+\alpha_2(v_1).
\end{equation*}
Given a vector subbundle $D_M\subset TM\oplus T^*M$, we denote by $D_M^\perp\subset TM\oplus T^*M$ the orthogonal of $D_M$ relative to this interior product. Observe that, for each $x\in M$, it is defined by
\begin{equation*}
D_M^\perp(x)=\left\{(v_1,\alpha_1)\in T_xM\times T_x^*M\mid\left\langle \!\left\langle(v_1,\alpha_1),(v_2,\alpha_2)\right\rangle \!\right\rangle=0,~\forall(v_2,\alpha_2)\in D_M(x)\right\}.
\end{equation*}

\begin{definition}
A \emph{Dirac structure}\footnote{The literature on Dirac geometry refers to them as \textit{almost Dirac structures}, and the term \textit{Dirac structure} is only used when an integrability condition is satisfied, see \cite{Co1990}. Since we shall not deal with the integrability condition in this paper, we shall call them \textit{Dirac structures}, following previous terminology in the Lagrange--Dirac and Hamilton--Dirac mechanics literature, see \cite{YoMa2006a, GBYo2015}.} on $M$ is a maximally isotropic vector subbundle $D_M\subset TM\oplus T^*M$ with respect to $\left\langle \!\left\langle\cdot,\cdot\right\rangle \!\right\rangle$, i.e., $D_M^\perp=D_M$.
\end{definition}

Note that a vector subbundle $D_M\subset TM\oplus T^*M$ is a Dirac structure if and only if $\dim D_M=\dim M$ and it satisfies
\begin{equation*}
\alpha_1(v_2)+\alpha_2(v_1)=0,\qquad(v_1,\alpha_1),(v_2,\alpha_2)\in D_M(x),~x\in M.
\end{equation*}
Given a regular distribution $\Delta_M\subset TM$ on $M$ and a two-form $\Omega_M\in\Omega^2(M)$ on $M$, a Dirac structure on $M$ is defined by, for each $x \in M$,
\begin{equation*}
D_M(x)=\left\{(v_x,\alpha_x)\in T_xM\times T_x^*M\mid v_x\in\Delta_M(x),~\alpha_x-\Omega_M^\flat(x)(v_x)\in\Delta_M^\circ(x)\right\},
\end{equation*}
where $\Delta_M^\circ\subset T^*M$ denotes the annihilator of $ \Delta _M$ and $\Omega_M^\flat: TM\to T^*M$ the flat map of $ \Omega _M$, which is given, for each $x\in M$, by $\Omega_M^\flat(x)(v_x)(w_x)=\Omega_M(x)(v_x,w_x)$ for any $v_x,w_x\in T_x M$.

\subsection{Finite-dimensional Lagrange--Dirac systems in mechanics}\label{sec:mechanicsprelim}

Let $Q$ be a finite-dimensional smooth manifold, playing the role of the \textit{configuration manifold in mechanics}. Let $\Omega_{T^*Q}\in\Omega^2(T^*Q)$ be the canonical symplectic two-form and $\Delta_Q\subset TQ$ be a regular distribution, playing the role of \emph{kinematic constraints}. We consider the lifted distribution, $\Delta_{T^*Q}=(T\pi_{Q})^{-1}(\Delta_Q)\subset T(T^*Q)$, where $\pi_Q: T^*Q \to Q$ is the canonical projection and $T\pi_Q: T\left(T^*Q\right) \to TQ$ is the tangent map of $\pi_Q$. The Dirac structure on $T^*Q$ that is induced from $\Delta_Q$ is given by, for each $p_q \in T^\ast Q$,
\begin{equation*}
\begin{split}
D_{\Delta_Q}(p_q)&=\left\{(v_{p_q},\alpha_{p_q})\in T_{p_q}(T^*Q)\times T^*_{p_q}(T^*Q)\mid v_{p_q}\in\Delta_{T^*Q}(p_q),\right. 
\\ &\left.\hspace{5cm}
\alpha_{p_q}-\Omega_{T^*Q}^\flat(p_q)(v_{p_q})\in\Delta_{T^*Q}^\circ(p_q) 
\right\}.
\end{split}
\end{equation*}
When $\Delta_Q=TQ$, it is known as the \emph{canonical Dirac structure} on $T^*Q$. 

Recall that there exists a structure of canonical isomorphisms between three iterated bundles over $Q$ that is called the {\it Tulczyjew triple}, and it is illustrated by the following diagram:
\begin{equation*}
	\begin{tikzpicture}
			\matrix (m) [matrix of math nodes,row sep=0.1em,column sep=7em,minimum width=2em]
			{	T^*(TQ) & T(T^*Q) & T^*(T^*Q)\\
			    (q,\delta q,\delta p,p) & (q,p,\delta q,\delta p) & (q,p,-\delta p,\delta q).\\};
			\path[-stealth]
			(m-1-1) edge [bend left = 25] node [above] {$ \gamma_{Q}$} (m-1-3)
			(m-1-2) edge [] node [above] {$\kappa_{Q}$} (m-1-1)
			(m-1-2) edge [] node [above] {$\Omega_{T^*Q}^\flat$} (m-1-3)
			(m-2-2) edge [|->] node [] {} (m-2-3)
			(m-2-2) edge [|->] node [] {} (m-2-1);
	\end{tikzpicture}
\end{equation*}

Let $L:TQ\to\mathbb R$ be a (possibly degenerate) Lagrangian. We denote by ${\mathrm d}_DL$ the \emph{Dirac differential} of $L$, which is a map given in coordinates by, for $(q,v) \in TQ$,
\begin{equation*}
{\mathrm d}_DL= \gamma_{Q}\circ{\rm d}L:TQ\to T^*(T^*Q),\quad (q,v) \mapsto \left( q, \frac{\partial L}{\partial v}, - \frac{\partial L}{\partial q}, v \right).
\end{equation*}

Consider an external force given by a fiber preserving map $F:TQ \rightarrow T^*Q$. The associated Lagrangian force field is the map $\widetilde{F}: TQ \rightarrow T^*(T^*Q)$ defined by, 
\begin{equation}\label{def_tilde_F} 
\left\langle \widetilde{F}(q,v), W \right\rangle = \left\langle F( q,v), T_{ \mathbb{F} L(q,v)} \pi _Q(W) \right\rangle, 
\end{equation} 
for $(q,v) \in TQ$ and $W \in T_{ \mathbb{F} L(q,v)}(T^*Q)$. Here $\mathbb{F} L:TQ \rightarrow T^*Q$ is the fiber derivative of $L$, locally given as $ \mathbb{F} L(q,v)=\left(q, \frac{\partial L}{\partial v}(q,v)\right)$. In coordinates the Lagrangian force field reads
\[
\widetilde{F}(q,v) = \left( q, \frac{\partial L}{\partial v}(q,v), F(q,v),0 \right).
\]

\begin{definition}\label{def:LDfinite}\rm
Given a Lagrangian $L:T Q \rightarrow \mathbb{R} $, a constraint $\Delta _Q$ on $Q$, and an external force $F:TQ \rightarrow T^*Q$,
the associated \emph{forced Lagrange--Dirac dynamical system} is given by $({\mathrm d}_DL, \widetilde{F}, D_{\Delta_Q})$ that satisfies the condition, for  each $(q,v,p) \in TQ\oplus T^*Q$,
\begin{equation}\label{DiracSys_Mech}
\left((q,p,\dot{q},\dot{p}),{\mathrm d}_DL(q,v)- \widetilde{F}(q,v)\right)\in D_{\Delta_Q}(q,p).
\end{equation}
\end{definition}

Observe that when $\Delta_Q$ is completely integrable, we have a \emph{holonomic} system. Otherwise, the system is \emph{nonholonomic}. The Lagrange--Dirac systems can be also obtained from the following variational structure.
\begin{definition}\label{LDAPPrin_Mechanics}\rm
The \emph{Lagrange--d'Alembert--Pontryagin principle} for curves $(q,v,p):[t_0,t_1]\to TQ\oplus T^*Q$ on the Pontryagin bundle is defined by the condition
\begin{equation}\label{LdAP}
\delta\int_{t_0}^{t_1}\left(L(q,v)+\left<p, \dot q-v \right>\right)dt  + \int_{t_0}^{t_1} \left\langle F(q, \dot  q), \delta q \right\rangle dt=0,
\end{equation}
with respect to variations $\delta q$, $ \delta v$, $ \delta  p$ such that $ \delta q\in\Delta_Q(q)$ and $\delta q(t_0)=\delta q(t_1)=0$, together with the condition $\dot  q \in \Delta _Q(q)$.
\end{definition}

The equivalence between the variational formulation \eqref{LdAP} and the forced Lagrange--Dirac system \eqref{DiracSys_Mech}, for curves $(q(t),v(t),p(t)) \in TQ \oplus T^*Q$ in the Pontryagin bundle is summarized in (iii)-(v) of the following theorem. One can also equivalently write these conditions on the curve $q(t) \in Q$ only, from which the curves $v(t)$ and $p(t)$ can be reconstructed. This is stated in (i)-(ii) of the following theorem, see \cite{YoMa2006a}.

\begin{theorem}\label{Theorem:LDfinitedim}\rm
Consider the curves $(q,v,p):[t_0,t_1]\to TQ\oplus T^* Q$ and $q=\rho_Q\circ(q,v,p): [t_{0},t_{1}] \to Q$, where $\rho_Q:TQ\oplus T^*Q\to Q$ is the natural projection. The following statements are equivalent:
\begin{itemize}
\item[(i)] The curve $q:[t_0,t_1]\to Q$ satisfies the critical condition for the action functional,
\[
\delta\int_{t_0}^{t_1}L(q(t),\dot q(t))\,dt   + \int_{t_0}^{t_1} \left\langle F(q(t), \dot  q(t)), \delta q(t) \right\rangle dt=0,
\]
where $\delta q(t),\dot q(t) \in \Delta_{Q}(q(t))$ for each $t\in[t_0,t_1]$.
\vskip 3pt
\item[(ii)] The curve $q:[t_0,t_1]\to Q$ is a solution of the \emph{Lagrange--d'Alembert equations},
\[
\frac{d}{dt}\frac{\partial L}{\partial \dot{q}}(q,\dot q)- \frac{\partial L}{\partial q}(q,\dot q)  -F(q, \dot  q)\in \Delta_{Q}^\circ(q).
\]
\vskip 3pt
\item[(iii)]
The curve $(q,v,p):[t_0,t_1]\to TQ\oplus T^*Q$ is critical for the Lagrange--d'Alembert--Pontryagin principle given in Definition \ref{LDAPPrin_Mechanics}.
\vskip 3pt
\item[(iv)]
The curve $(q,v,p):[t_0,t_1]\to TQ\oplus T^*Q$  is a solution of the \emph{Lagrange--d'Alembert--Pontryagin equations},
\begin{equation*}
p=\frac{\partial L}{\partial v}(q,v),\qquad v=\dot q\in\Delta_Q(q),\qquad \dot{p}-\frac{\partial L}{\partial q}(q,v)- F(q, \dot  q)\in\Delta_Q^\circ(q).
\end{equation*}
\vskip 3pt
\item[(v)]
The curve $(q,v,p):[t_0,t_1]\to TQ\oplus T^*Q$ is a solution of the forced Lagrange--Dirac system given in Definition \ref{def:LDfinite}, i.e.,
\[
\left((q(t),p(t),\dot{q}(t),\dot{p}(t)),{\mathrm d}_DL(q(t),v(t))- \tilde F(q(t),v(t))\right)\in D_{\Delta_Q}(q(t),p(t)).
\]
\end{itemize}
\end{theorem}

In this paper we shall formulate extensions of this theorem to the infinite-dimensional case, allowing the treatment of systems with body and boundary external forces, such as control forces. Specifically, we will not consider any constraints in this paper, which corresponds to the case $ \Delta _Q=TQ$.

\subsection{Fr\'{e}chet spaces and dual systems}\label{sec:frechet}

We recall that a Fr\'echet space is a locally convex topological vector space, whose topology is induced by a complete invariant metric.
In this work, we focus on the Fr\'{e}chet space $V=C^\infty(\mathcal B)$ of smooth functions on a bounded domain $\mathcal B\subset\mathbb R^m$ with smooth boundary and, more generally, on the space $V=\Omega^k(M)$ of smooth $k$-forms on a finite-dimensional, compact smooth manifold $M$ with smooth boundary.
The Fréchet topology of $C^\infty(\mathcal B)$ is the final topology induced by the map $C^\infty(\mathbb R^n)\ni f\mapsto f|_{\mathcal B}\in C^\infty(\mathcal B)$, i.e.,  the finest topology on $C^\infty(\mathcal B)$ that makes the previous map continuous.
In turn, the Fr\'{e}chet topology of $C^\infty(\mathbb R^m)$ is defined by the following family of seminorms:
\begin{equation}\label{eq:seminorms}
p_n(\varphi)=\max\{|\partial_\alpha\varphi(x)|: x\in K_n,~|\alpha|\leq n\},\qquad\varphi\in C^\infty(\mathbb R^m),~n\in\mathbb N_0,
\end{equation}
where $\{K_n\subset\mathbb R^m\mid K_n\subset K_{n+1}^\circ,~n\in\mathbb N\}$ is a family of compact sets such that $\mathbb R^m=\bigcup_{n=1}^\infty K_n$, $\alpha=(\alpha_1,\dots,\alpha_m)\in\mathbb N_0^m$ is a multi-index, $|\alpha|=\alpha_1+\dots+\alpha_m$ denotes its length and $\partial_\alpha=\prod_{i=1}^m(\partial/\partial x^i)^{\alpha_i}$, being $x=(x^1,\dots,x^m)$ the standard (global) coordinates on $\mathbb R^m$.

Given a Fréchet space, we distinguish two dual spaces:
\begin{enumerate}
    \item \emph{Algebraic dual}, the space of linear functions from $V$ to the field, $\mathbb R$;
    \item \emph{Topological dual}, the space of linear and continuous functions from $V$ to the field, $\mathbb R$.
\end{enumerate}
Throughout the paper,  we will denote the latter by $V'$. Naturally, both dual spaces agree when $V$ is finite-dimensional. Observe that the pair $(V,V')$ is a \emph{dual system} in the sense of \cite[Chapter 23]{MeMeVoRa1997}, that is, $V'$ is a subspace of the algebraic dual and it is in weak non-degenerate duality with $V$, i.e., the condition $\alpha(\varphi)=0$ for each $\alpha\in V'$ implies that $\varphi=0$.
In addition, $V'$ is naturally endowed with the \emph{weak topology} \cite[\S 3.11]{Ru1991}, i.e., the weakest (smallest) topology on $V'$ such that $\langle\cdot,\varphi\rangle:V'\to\mathbb R$ is continuous for each $\varphi\in V$, where $\langle\cdot,\cdot\rangle$ denotes the duality pairing.

For the space $V=C^\infty(\mathcal{B})$ of smooth functions on a bounded domain, the topological dual is the space of \emph{supported distributions}, $C^\infty(\mathcal B)'=\dot C^{-\infty}(\mathcal B)$, which consists of (standard) distributions on any compact extension\footnote{A compact extension of $\mathcal B$ is any boundaryless, compact manifold $\tilde{\mathcal B}$ such that $\mathcal B\subset\tilde{\mathcal B}$ is a submanifold. For instance, one may consider the \emph{double copy} construction: $\tilde{\mathcal B}=(\mathcal B\sqcup\mathcal B)/\partial\mathcal B$, the disjoint union of two copies of the domain with the points on the boundary identified.} of $\mathcal B$ whose support is contained in $\mathcal B$.

As we will see below (cf. \S\ref{sec:LDfunctions} and \S\ref{sec:LDkforms}), the fiber derivatives of Lagrangians defined through a density lie in some subspace of the topological dual. For this reason, it is convenient to work with another dual system, $(V,V^\star)$, where $V^\star\subset V'$ is a linear subspace that will be called \emph{restricted dual space}. A key point in our development is the definition of the canonical symplectic form on $T^\star V=V\times V^\star$ (cf. \S\ref{sec:tulczyjewfunctions} and \S\ref{sec:tulczyjewkforms} for more details), which will be denoted by $\Omega_{T^\star V}\in\Omega^2(T^\star V)$. As we will show, the range of the corresponding flat map, $\Omega_{T^\star V}^\flat:T(T^\star V)\to T'(T^\star V)$, is the vector subspace given by $\operatorname{Im}\Omega_{T^\star V}^\flat=V\times V^\star\times V^\star\times V\subset V\times V^\star\times V'\times(V^\star)'=T'(T^\star V)$. In general, the inclusion $V^\star\subset V'$ is strict and, thus, $\Omega_{T^\star V}$ is a weak form (cf. Remark \ref{remark:weakvsstrong}).

\section{Lagrange--Dirac dynamical systems on the space of smooth functions}\label{sec:functions}

In this section, we develop the theory for infinite-dimensional Lagrange--Dirac dynamical systems whose configuration manifold is given by
\begin{equation*}
V=C^\infty(\mathcal B),
\end{equation*}
where $\mathcal B\subset\mathbb R^m$ is a bounded domain with smooth boundary. Henceforth, ${dx}$ and ${ds}$ will denote the volume form on $\mathcal B$ and the area element on $\partial\mathcal B$, respectively. The extension to the case where $ \mathcal{B} $ is a compact manifold with boundary is treated in \S\ref{sec:kforms} for the more general case of the space of $k$-forms.

\subsection{Restricted dual and restricted cotangent bundle}\label{sec:restricteddualfunctions}

As mentioned above, the partial derivatives of a Lagrangian defined through a density lie in a vector subspace of $V'$, which motivates the definition of the restricted dual space.

\begin{definition}\rm The \emph{restricted dual} of $V=C^\infty(\mathcal B)$ is defined as
\begin{equation*}
V^\star=C^\infty(\mathcal B)\times C^\infty(\partial\mathcal B)
\end{equation*}
and it is endowed with the product topology.
\end{definition}

The following standard result plays a main role in our approach, therefore we state it as a proposition.

\begin{proposition}\label{prop:restricteddualfunctions}\rm
The restricted dual $V^\star$ is a Fr\'{e}chet space and the map
\begin{equation}\label{eq:identificationrestricted}
\Psi:V^\star\to V',\quad\left\langle\Psi(\alpha,\alpha_\partial),\varphi\right\rangle=\int_{\mathcal B}\alpha\varphi\,{dx}+\int_{\partial\mathcal B}\,\alpha_\partial\varphi\,{ds},
\end{equation}
for each $\varphi\in V$ and $(\alpha,\alpha_\partial)\in V^\star$, is a continuous injection.
\end{proposition}

\begin{proof}
It is clear that the topology on $V^\star$ is Fr\'{e}chet, as the product of Fréchet spaces is a Fréchet space. The injectivity is a straightforward consequence of the Fundamental Lemma of the Calculus of Variations, which also gives the (weak) non-degeneracy of pairing. In order to prove the continuity, let $(\alpha^n,\alpha_\partial^n)_{n=1}^\infty$ be a sequence in $V^\star$ such that $(\alpha^n,\alpha_\partial^n)\overset{n}{\to}(\alpha,\alpha_\partial)\in V^\star$ with the product of the Fréchet topologies. Note that the compactness of $\mathcal B$ ensures that there exists $N\in\mathbb Z^+$ such that $\mathcal B\subset K_N$, where we are following the notations in \eqref{eq:seminorms}. Subsequently, the convergence of the previous sequence yields $p_N(\alpha^n-\alpha)\overset{n}{\to}0$ and $p_N(\alpha_\partial^n-\alpha_\partial)\overset{n}{\to}0$. In turn, this implies that $(\alpha^n,\alpha_\partial^n)_{n=1}^\infty$ uniformly converges to $(\alpha,\alpha_\partial)$. On the other hand, by the definition of weak topology, the sequence $(\Psi(\alpha^n,\alpha_\partial^n))_{n=1}^\infty$ is convergent to $\Psi(\alpha,\alpha_\partial)$ in $V'$ if and only if
\begin{equation}\label{eq:weakconvergence}
\langle\Psi(\alpha^n,\alpha_\partial^n),\varphi\rangle\overset{n}{\to}\langle\Psi(\alpha,\alpha_\partial),\varphi\rangle,\qquad\forall\varphi\in C^\infty(\mathcal B).
\end{equation}
Let us check this condition: let $\varphi\in C^\infty(\mathcal B)$, then
\begin{align*}
|\langle\Psi(\alpha^n,\alpha_\partial^n),\varphi\rangle-\langle\Psi(\alpha,\alpha_\partial),\varphi\rangle| & =\left|\int_{\mathcal B}\alpha^n\varphi\,{dx}+\int_{\partial\mathcal B}\,\alpha_\partial^n\varphi\,{ds}-\int_{\mathcal B}\alpha\varphi\,{dx}-\int_{\partial\mathcal B}\,\alpha_\partial\varphi\,{ds}\right|\\
& \leq\left|\int_{\mathcal B}(\alpha^n-\alpha)\varphi\,{dx}\right|+\left|\int_{\partial\mathcal B}(\alpha_\partial^n-\alpha_\partial)\varphi\,{ds}\right|\\
& \leq\int_{\mathcal B}|\alpha^n-\alpha|\cdot|\varphi|\,{dx}+\int_{\partial\mathcal B}|\alpha_\partial^n-\alpha_\partial|\cdot|\varphi|\,{ds}\\
& \leq K_1\,\max_{x\in\mathcal B}\{|\alpha^n(x)-\alpha(x)|\}+K_2\,\max_{x\in\partial\mathcal B}\{|\alpha_\partial^n(x)-\alpha_\partial(x)|\},
\end{align*}
where $K_1=\operatorname{vol}\mathcal B\cdot\max_{x\in\mathcal B}\{|\varphi(x)|\}$ and $K_2=\operatorname{vol}\partial\mathcal B\cdot\max_{x\in\partial\mathcal B}\{|\varphi(x)|\}$. Therefore, the uniform convergence of $(\alpha^n,\alpha_\partial^n)_{n=1}^\infty$ to $(\alpha,\alpha_\partial)$ ensures that \eqref{eq:weakconvergence} is satisfied and we conclude.
\end{proof}

\begin{remark}\rm Note that $V^\star$ is not a closed subspace of $V'$. In particular, the map $\Psi$ is not closed.
Indeed, by contradiction suppose that for each sequence $(\alpha^n,\alpha_\partial^n)_{n=1}^\infty$ in $V^\star$ such that $(\Psi(\alpha^n,\alpha_\partial^n))_{n=1}^\infty$ is convergent to some $\psi\in V'$, then $\psi\in\Psi(V^\star)$, i.e., there exists $(\alpha,\alpha_\partial)\in V^\star$ such that $\psi=\Psi(\alpha,\alpha_\partial)$. Hence, the weak convergence condition \eqref{eq:weakconvergence} yields
\begin{equation*}
\int_{\mathcal B}|\alpha^n-\alpha|\,{dx}\overset{n}{\to}0,\qquad\int_{\partial\mathcal B}\,|\alpha_\partial^n-\alpha_\partial|\,{ds}\overset{n}{\to}0.
\end{equation*}
Then $(\alpha^n,\alpha_\partial^n)\overset{n}{\to}(\alpha,\alpha_\partial)$ pointwisely (up to a subsequence). Nevertheless, we have the following counterexample. Let $\mathcal B=[-1,1]\subset\mathbb R$ and $(\alpha^n,\alpha_\partial^n)_{n=1}^\infty=(\sqrt{x^2+1/n},0)_{n=1}^\infty$ in $V^\star=C^\infty([-1,1])\times\mathbb R^2$, where we have used that $\partial[-1,1]=\{-1,1\}$ and $C^\infty(\{-1,1\})=\mathbb R^2$. It is clear that the previous sequence pointwisely converges to $(\alpha,\alpha_\partial)=(|x|,0)\notin C^\infty([-1,1])\times\mathbb R^2$, while $\Psi(\alpha^n,\alpha_\partial^n)\overset{n}{\to}\Psi(\alpha,\alpha_\partial)$ in $V'$ since
\begin{equation*}
\langle\Psi(\alpha^n,\alpha_\partial^n)-\Psi(\alpha,\alpha_\partial),\varphi\rangle\leq 2K_\varphi\int_0^1\left(\sqrt{x^2+\frac{1}{n}}-x\right){dx}\overset{n}{\to}0,\quad\forall\;\varphi\in V,
\end{equation*}
where $K_\varphi=\max_{x\in[-1,1]}\{\varphi(x)\}$.
\end{remark}

Thanks to the previous proposition, in the following we identify $V^\star\ni(\alpha,\alpha_\partial)\simeq\Psi(\alpha,\alpha_\partial)\in V'$, thus regarding the restricted dual $V^\star$ as a subspace of $V'$. Note that the $L^2$-pairing given above is  weakly non-degenerate, hence $(V, V^\star)$ is a dual system. Since $V$ is a vector space, its tangent and cotangent bundles are trivial,
\begin{equation*}
TV=V\times V,\qquad T'V=V\times V'.
\end{equation*}
The \emph{restricted cotangent bundle} is defined to be
\begin{equation*}
T^\star V=V\times V^\star\subset T'V,
\end{equation*}
where the inclusion is understood through the identification \eqref{eq:identificationrestricted}, which makes $T^\star V$ a subbundle\footnote{In this paper by a subbundle of $T'V=V \times V'$ we simply mean a space of the form $V \times E$ with $E \subset V'$ a topological vector space continuously embedded in $V'$.} of $T'V$. The iterated bundles of $TV$ read
\begin{equation*}
T(TV)=V\times V\times V\times V,\qquad T'(TV)=V\times V\times V'\times V'.
\end{equation*}
Analogously, the \emph{restricted cotangent bundle} of $TV$ is defined as
\begin{equation*}
T^\star(TV)=V\times V\times V^\star\times V^\star\subset T'(TV).
\end{equation*}
At last, the iterated bundles of the restricted cotangent bundle are given by
\begin{equation*}
T(T^\star V)=V\times V^\star\times V\times V^\star,\qquad T'(T^\star V)=V\times V^\star\times V'\times(V^\star)',
\end{equation*}
where $(V^\star)'$ is the topological dual of $V^\star$. It is clear that $T(T^\star V)\subset T(T'V)$. On the other hand, note that we may regard $V\subset(V^\star)'$ by means of the $L^2$-pairing. This is due to the fact that $(V^\star,V)$ is a dual system too (cf. \cite[Chapter 23]{MeMeVoRa1997}). More specifically, this inclusion is given by the following assignment:
\begin{equation}\label{eq:idenficationdualrestricted}
V\ni\varphi\mapsto F_\varphi\in(V^\star)',
\end{equation}
where $F_\varphi(\alpha,\alpha_\partial)=\left\langle(\alpha,\alpha_\partial),\varphi\right\rangle$ for each $(\alpha,\alpha_\partial)\in V^\star$. The image of the map $F$, denoted $(V^\star)^\star\subset (V^\star)'$, is canonically identified with $V$.

Again, for simplicity we denote both objects by the same symbol, $\varphi\simeq F_\varphi$. Therefore, if we define the restricted cotangent bundle of $T^\star V$ as
\begin{equation*}
T^\star(T^\star V)=V\times V^\star\times V^\star\times V,    
\end{equation*}
where $(V^{\star})^{\star}\cong V$, it may be regarded as a subbundle of $T'(T^\star V)$.

In short, the restricted iterated bundles are given by
\begin{equation*}
\begin{array}{l}
T^\star(TV)=V\times V\times V^\star\times V^\star \subset T'(TV),\vspace{0.1cm}\\
T(T^\star V)=V\times V^\star\times V\times V^\star \subset T(T'V) ,\vspace{0.1cm}\\
T^\star(T^\star V)=V\times V^\star\times V^\star\times V\subset T'(T^\star V).
\end{array}
\end{equation*}
To conclude, observe that the topological Pontryagin bundle of the restricted cotangent bundle reads
\begin{equation*}
T(T^\star V)\oplus T'(T^\star V)= V\times V^\star\times
\big((V\times V^\star)\,\oplus \,(V'\times(V^\star)')\big).
\end{equation*}
The \emph{restricted Pontryagin bundle} of $T^\star V$ is a vector subbundle of the topological Pontryagin bundle defined as
\begin{equation*}
T(T^\star V)\oplus T^\star(T^\star V)=V\times V^\star\times\big((V\times V^\star) \,\oplus\, (V^\star\times V)\big).
\end{equation*}

\subsection{Canonical forms, Tulczyjew triple and the canonical Dirac structure}\label{sec:tulczyjewfunctions}

Now, using the restricted dual, the duality pairing introduced in Proposition \ref{prop:restricteddualfunctions} and the restricted iterated bundles, we define the canonical forms, the Tulczyjew triple and the canonical Dirac structure on $T^{\star}V=T^{\star}C^\infty(\mathcal B)$.
While we can use the usual definitions for all these objects, our choice of restricted dual space induces boundary terms that will play a crucial role in describing systems with energy boundary flow.

\begin{definition}\rm\label{DefCanForms}
The \emph{canonical one-form} on $T^\star V$, $\Theta_{T^\star V}\in\Omega^1(T^\star V)$, is defined as
$$
\Theta_{T^\star V}(z) \cdot \delta{z}=\left< z, T_{z}\pi_{V}(\delta z)\right>,\qquad z \in T^\star V,~\delta z\in T_{z}\left(T^\star V\right),
$$
where $\pi_{V}: T^\star V \to V$. Furthermore, the \emph{canonical symplectic two-form} on $T^\star V$, $\Omega_{T^\star V}\in\Omega^2(T^\star V)$, is defined as $\Omega_{T^\star V}=-{\rm d}\Theta_{T^\star V}$.
\end{definition}

Note that $ \Theta _{T^\star V}$ is a smooth form on the Fr\'echet space $T^\star V$, so that the exterior derivative can be computed in the usual sense.
Since $z=(\varphi,\alpha,\alpha_{\partial})\in T^\star V$, $\delta{z}=(\delta\varphi,\delta\alpha,\delta\alpha_{\partial}) \in T_{z}\left(T^\star V\right)\simeq V\times V^\star$, one gets
\begin{equation*}
\Theta_{T^\star V}(\varphi,\alpha,\alpha_\partial)\cdot(\delta\varphi,\delta\alpha,\delta\alpha_{\partial})=\left< (\varphi, \alpha,\alpha_\partial),\delta\varphi\right>=\int_{\mathcal B} \alpha \delta{\varphi} \,{dx} + \int_{\partial \mathcal B} \alpha_{\partial} \delta{\varphi}\,{ds}.
\end{equation*}
Then, the canonical symplectic form $\Omega_{T^\star V}=-{\rm d}\Theta_{T^\star V}$ is given by
\begin{align*}
\Omega_{T^\star V}(\varphi,\alpha,\alpha_\partial)((\dot\varphi,\dot\alpha,\dot\alpha_\partial),(\delta\varphi,\delta\alpha,\delta\alpha_\partial))&=\langle(\delta\alpha,\delta\alpha_\partial),\dot\varphi\rangle-\langle(\dot\alpha,\dot\alpha_\partial),\delta\varphi\rangle\\
&=\int_{\mathcal{B}} ( \delta\alpha \dot \varphi - \dot\alpha\delta\varphi) dx +\int_{\partial\mathcal{B}} (\delta\alpha_\partial  \dot \varphi - \dot\alpha_\partial \delta\varphi) ds,
\end{align*}
for each $(\varphi,\alpha,\alpha_\partial)\in T^\star V$ and $(\dot\varphi,\dot\alpha,\dot\alpha_\partial),~(\delta\varphi,\delta\alpha,\delta\alpha_\partial)\in T_{(\varphi,\alpha,\alpha_\partial)}(T^\star V)$.

\begin{proposition}\label{prop:Omegaflatfunctions}\rm
The flat map of the canonical symplectic form defines a vector bundle isomorphism over the identity, $\operatorname{id}_{T^\star V}$, between $T(T^\star V)$ and the restricted iterated bundle, $T^\star(T^\star V) \subset T'(T^\star V)$. Furthermore, under the identification of Proposition \ref{prop:restricteddualfunctions}, for each $(\varphi,\alpha,\alpha_\partial)\in T^\star V$, it is given by
\begin{equation*}
\Omega_{T^\star V}^\flat(\varphi,\alpha,\alpha_\partial):T_{(\varphi,\alpha,\alpha_\partial)}(T^\star V)\to T^\star_{(\varphi,\alpha,\alpha_\partial)}(T^\star V),\quad(\dot\varphi,\dot\alpha,\dot\alpha_\partial)\mapsto( -\dot\alpha,-\dot\alpha_\partial,\dot\varphi).
\end{equation*}
\end{proposition}
\begin{proof}
By definition of flat map, for each $(\dot\varphi,\dot\alpha,\dot\alpha_\partial),(\delta\varphi,\delta\alpha,\delta\alpha_\partial)\in T_{(\varphi,\alpha,\alpha_\partial)}(T^\star V)$, we have
\begin{align*}
\left<\Omega_{T^\star V}^\flat(\varphi,\alpha,\alpha_\partial)(\dot\varphi,\dot\alpha,\dot\alpha_\partial),\; (\delta\varphi,\delta\alpha,\delta\alpha_\partial)\right>
&=\Omega_{T^\star V}(\varphi,\alpha,\alpha_\partial)\left((\dot\varphi,\dot\alpha,\dot\alpha_\partial),(\delta\varphi,\delta\alpha,\delta\alpha_\partial)\right)\\
&=\langle(\delta\alpha,\delta\alpha_\partial),\dot\varphi\rangle-\langle(\dot\alpha,\dot\alpha_\partial),\delta\varphi\rangle.
\end{align*}
Since the previous expression is valid for every $(\delta\varphi,\delta\alpha,\delta\alpha_\partial)\in T_{(\varphi,\alpha,\alpha_\partial)}(T^\star V)$ and by taking the identifications \eqref{eq:identificationrestricted} and \eqref{eq:idenficationdualrestricted} into account, we may write
\begin{equation*}
\Omega_{T^\star V}^\flat(\varphi,\alpha,\alpha_\partial) (\dot\varphi,\dot\alpha,\dot\alpha_\partial)=(-\dot\alpha,-\dot\alpha_\partial,\dot\varphi)\in T^\star_{(\varphi,\alpha,\alpha_\partial)}(T^\star V)\subset T'_{(\varphi,\alpha,\alpha_\partial)}(T^\star V).
\end{equation*}
To conclude, note that $\Omega_{T^\star V}^\flat$ is continuous, and its image is given by $\operatorname{Im}\Omega_{T^\star V}^\flat=T^\star(T^\star V)$. Hence, it is clear that its inverse $(\Omega_{T^\star V}^\flat)^{-1}: T^\star(T^\star V) \to T(T^\star V)$ is given by
\begin{equation*}
T^\star(T^\star V)\ni(\varphi,\alpha,\alpha_\partial,\,\dot\alpha,\dot\alpha_\partial,\dot\varphi)\mapsto(\varphi,\alpha,\alpha_\partial,\, \dot\varphi,-\dot\alpha,-\dot\alpha_\partial)\in T(T^\star V),
\end{equation*}
which is also continuous.
\end{proof}

\begin{remark}[$ \Omega _{T^\star V}$ as a strong form on the restricted  duals]\label{remark:weakvsstrong}\rm
Observe that the inclusion $T^\star(T^\star V)\subset T'(T^\star V)$ is strict.
Therefore, the canonical symplectic form, $\Omega_{T^\star {V}}$, is weak, since it does not define an isomorphism between $T(T^\star V)$ and $T'(T^\star V)$. If we confine ourselves to the restricted iterated bundle, then it becomes a strong form. In the following, we focus on the latter situation without further mention.
\end{remark}

\begin{definition}\rm
By mimicking the finite-dimensional case, we define the following canonical isomorphism over the identity, $\operatorname{id}_V$,
\begin{equation*}
\kappa_{T^\star V}:T(T^\star V)\to T^\star(TV),\quad(\varphi,\alpha,\alpha_\partial, \dot\varphi,\dot\alpha,\dot\alpha_\partial)\mapsto(\varphi,\dot\varphi, \dot\alpha,\dot\alpha_\partial,\alpha,\alpha_\partial).
\end{equation*}
In the same vein, we set $\gamma_{T^\star V}=\Omega_{T^\star V}^\flat\circ\kappa_{T^\star V}^{-1}$, which is explicitly given by
\begin{equation*}
\gamma_{T^\star V}:T^\star(TV)\to T^\star(T^\star V),\quad(\varphi,\dot\varphi, \alpha,\alpha_\partial,\dot\alpha,\dot\alpha_\partial)\mapsto(\varphi,\dot\alpha,\dot\alpha_\partial, -\alpha,-\alpha_\partial,\dot\varphi).
\end{equation*}

By gathering the previous isomorphisms, we obtain the \emph{restricted Tulczyjew triple} on the space of smooth functions as follows (compare with the analog diagram in \S\ref{sec:mechanicsprelim}):
\begin{equation*}
	\begin{tikzpicture}
			\matrix (m) [matrix of math nodes,row sep=0.1em,column sep=6em,minimum width=2em]
			{	T^\star(TV)) & T(T^\star V) & T^\star(T^\star V))\\
			\left(\varphi,\dot\varphi, \dot\alpha,\dot\alpha_\partial,\alpha,\alpha_\partial\right) & \left(\varphi,\alpha,\alpha_\partial, \dot\varphi,\dot\alpha,\dot\alpha_\partial\right) & \left(\varphi,\alpha,\alpha_\partial, -\dot\alpha,-\dot\alpha_\partial,\dot\varphi\right).\\};
			\path[-stealth]
			(m-1-1) edge [bend left = 25] node [above] {$\gamma_{T^\star V}$} (m-1-3)
			(m-1-2) edge [] node [above] {$\kappa_{T^\star V}$} (m-1-1)
			(m-1-2) edge [] node [above] {$\Omega_{T^\star V}^\flat$} (m-1-3)
			(m-1-3)
			(m-2-2) edge [|->] node [] {} (m-2-3)
			(m-2-2) edge [|->] node [] {} (m-2-1);
	\end{tikzpicture}
\end{equation*}

\end{definition}

\begin{proposition}\label{prop:diracstructurefunctions}\rm
The subbundle $D_{T^\star V} =\operatorname{graph}\,\Omega_{T^\star V}^\flat$ of the restricted Pontryagin bundle of $T(T^\star V)\oplus T^\star(T^\star V)$ is a Dirac structure on  $T^\star V$, which is called the \emph{canonical Dirac structure} on $T^\star V$.
\end{proposition}

\begin{proof}
Let $(\varphi,\alpha,\alpha_\partial)\in T^\star V$. By using the expression of $\Omega_{T^\star  V}^\flat$ given in Proposition \ref{prop:Omegaflatfunctions}, we have
\begin{equation}\label{eq:DVfunctions}
\begin{aligned}
&D_{T^\star V}(\varphi,\alpha,\alpha_\partial)=\big\{(\dot\varphi,\dot\alpha,\dot\alpha_\partial, \delta\alpha,\delta\alpha_\partial,\delta\varphi)\in T_{(\varphi,\alpha,\alpha_\partial)}(T^\star V)\times T_{(\varphi,\alpha,\alpha_\partial)}^\star(T^\star V)\mid\\
& \hspace{9cm} -\dot\alpha=\delta\alpha,~-\dot\alpha_\partial=\delta\alpha_\partial,~\dot\varphi=\delta\varphi\big\},
\end{aligned} 
\end{equation}
and $D_{T^\star V}^\perp(\varphi,\alpha,\alpha_\partial)$ is given by
\begin{multline*}
D_{T^\star V}^\perp(\varphi,\alpha,\alpha_\partial)=\big\{\big(\dot\phi,\dot\beta,\dot\beta_\partial, \delta\beta,\delta\beta_\partial,\delta\phi\big)\in T_{(\varphi,\alpha,\alpha_\partial)}(T^\star V)\times T_{(\varphi,\alpha,\alpha_\partial)}^\star(T^\star V)\mid\\
\big\langle \! \big\langle \big(\dot\varphi,\dot\alpha,\dot\alpha_\partial, \delta\alpha,\delta\alpha_\partial,\delta\varphi\big),\big(\dot\phi,\dot\beta,\dot\beta_\partial, \delta\beta,\delta\beta_\partial,\delta\phi\big)\big\rangle \! \big\rangle =0,\\
\forall(\dot\varphi,\dot\alpha,\dot\alpha_\partial, \delta\alpha,\delta\alpha_\partial,\delta\varphi)\in D_{T^\star V}(\varphi,\alpha,\alpha_\partial)\big\}.
\end{multline*}

Let $\big(\dot\varphi,\dot\alpha,\dot\alpha_\partial, \delta\alpha,\delta\alpha_\partial,\delta\varphi\big),\big(\dot\phi,\dot\beta,\dot\beta_\partial, \delta\beta,\delta\beta_\partial,\delta\phi\big)\in D_{T^\star  V}(\varphi,\alpha,\alpha_\partial)$. From \eqref{eq:DVfunctions}, we have
\begin{multline*}
\big\langle \!\big\langle \big(\dot\varphi,\dot\alpha,\dot\alpha_\partial, \delta\alpha,\delta\alpha_\partial,\delta\varphi\big),\big(\dot\phi,\dot\beta,\dot\beta_\partial, \delta\beta,\delta\beta_\partial,\delta\phi\big) \big\rangle \! \big\rangle \\
=\big\langle\big(\delta\alpha,\delta\alpha_\partial\big),\dot\phi\big\rangle+\big\langle\big(\dot\beta,\dot\beta_\partial\big),\delta\varphi\big\rangle+\big\langle\left(\delta\beta,\delta\beta_\partial\right),\dot\varphi\big\rangle+\big\langle\big(\dot\alpha,\dot\alpha_\partial\big),\delta\phi\big\rangle\\
=-\big\langle\big(\dot\alpha,\dot\alpha_\partial\big),\dot\phi\big\rangle+\big\langle\big(\dot\beta,\dot\beta_\partial\big),\dot\varphi\big\rangle-\big\langle\big(\dot\beta,\dot\beta_\partial\big),\dot\varphi\big\rangle+\big\langle\big(\dot\alpha,\dot\alpha_\partial\big),\dot\phi\big\rangle=0.
\end{multline*}
Therefore, $D_{T^\star V}\subset D_{T^\star V}^\perp$.
\medskip

Reciprocally, let $\big(\dot\phi,\dot\beta,\dot\beta_\partial, \delta\beta,\delta\beta_\partial,\delta\phi\big)\in D_{T^\star V}^\perp(\varphi,\alpha,\alpha_\partial)$. Choose  
\begin{equation*}
(\dot\varphi,\dot\alpha,\dot\alpha_\partial, \delta\alpha,\delta\alpha_\partial,\delta\varphi)=(0,\dot\alpha,\dot\alpha_\partial, -\dot\alpha,-\dot\alpha_\partial,0)\in D_{T^\star V}(\varphi,\alpha,\alpha_\partial),\qquad(\dot\alpha,\dot\alpha_\partial)\in V^\star.
\end{equation*}
Then,
\begin{equation*}
\begin{split}
0&=\big\langle\!\big\langle(0,\dot\alpha,\dot\alpha_\partial, -\dot\alpha,-\dot\alpha_\partial,0),\big(\dot\phi,\dot\beta,\dot\beta_\partial, \delta\beta,\delta\beta_\partial,\delta\phi\big)\big\rangle\!\big\rangle\\
&\hspace{3cm}=-\big\langle\big(\dot\alpha,\dot\alpha_\partial\big),\dot\phi\big\rangle+\big\langle\big(\dot\alpha,\dot\alpha_\partial\big),\delta\phi\big\rangle=\big\langle\big(\dot\alpha,\dot\alpha_\partial\big),\delta\phi-\dot\phi\big\rangle.
\end{split}
\end{equation*}
Since this is valid for each $(\dot\alpha,\dot\alpha_\partial)\in V^\star$, we get $\delta\phi=\dot\phi$. Analogously, choose
\begin{equation*}
(\dot\varphi,\dot\alpha,\dot\alpha_\partial, \delta\alpha,\delta\alpha_\partial,\delta\varphi)=(\dot\varphi,0,0, \, 0,0,\dot\varphi)\in D_{T^\star V}(\varphi,\alpha,\alpha_\partial),\qquad\dot\varphi\in V.
\end{equation*}
Then, since
\begin{equation*}
\begin{split}
0&=\big\langle\!\big\langle(\dot\varphi,0,0,\, 0,0,\dot\varphi),\big(\dot\phi,\dot\beta,\dot\beta_\partial,\, \delta\beta,\delta\beta_\partial,\delta\phi\big)\big\rangle\!\big\rangle\\
&\hspace{3cm}=\big\langle\big(\dot\beta,\dot\beta_\partial\big),\dot\varphi\big\rangle+\big\langle\big(\delta\beta,\delta\beta_\partial\big),\dot\varphi\big\rangle=\big\langle\big(\dot\beta+\delta\beta,\dot\beta_\partial+\delta\beta_\partial\big),\dot\varphi\big\rangle
\end{split}
\end{equation*}
holds for each $\dot\varphi\in V$, we obtain $\delta\beta=-\dot\beta$ and $\delta\beta_\partial=-\dot\beta_\partial$. Since $\big(\dot\phi,\dot\beta,\dot\beta_\partial,\,\delta\beta,\delta\beta_\partial,\delta\phi\big)\in D_{T^\star V}(\varphi,\alpha,\alpha_\partial)$, we conclude $D_{T^\star V}^\perp \subset  D_{T^\star V}$. Therefore, $D_{T^\star V}^\perp= D_{T^\star V}$ as desired.
\end{proof}

\subsection{Infinite-dimensional Lagrange--Dirac dynamical systems}\label{sec:LDfunctions}
Consider a Lagrangian defined through a density, i.e.,
\begin{equation}\label{eq:lagrangianfunctions}
L:TV\to\mathbb R,\quad L(\varphi,\nu)=\int_{\mathcal B}\mathfrak L(\varphi(x),\nu(x),\nabla\varphi(x))\,{dx},
\end{equation}
where $\mathfrak L:\mathbb R\times\mathbb R\times\mathbb R^m\to\mathbb R$ is the Lagrangian density, and
\begin{equation*}
\nabla:C^\infty(\mathcal B)\to\mathfrak X(\mathcal B)\simeq C^\infty(\mathcal B,\mathbb R^m),\quad\varphi\mapsto\left(\frac{\partial\varphi}{\partial x^1},\dots,\frac{\partial\varphi}{\partial x^m}\right),
\end{equation*}
is the gradient, with $(x^1,\dots,x^m)$ the standard (global) coordinates of $\mathbb R^m$.

The partial functional  derivatives of the Lagrangian are the maps
\begin{equation}\label{def_delta_L}
\frac{\delta L}{\delta\varphi},~\frac{\delta L}{\delta \nu }:TV\to V',
\end{equation}
defined, for each $(\varphi, \nu)\in TV$, by
\begin{equation*}
\frac{\delta L}{\delta\varphi}(\varphi,\nu)(\delta\varphi)=\left.\frac{d}{d\epsilon}\right|_{\epsilon =0}L(\varphi+\epsilon\,\delta\varphi,\nu),\quad\frac{\delta L}{\delta \nu }(\varphi,\nu)(\delta\nu)=\left.\frac{d}{d\epsilon}\right|_{\epsilon=0}L(\varphi,\nu+\epsilon\,\delta\nu),
\end{equation*}
where $\delta\varphi, \delta\nu\in V$. In the following, the contraction between $\omega=\left(\omega_1,\dots,\omega_m\right)\in\Omega^1(\mathcal B)\simeq C^\infty(\mathcal B,\mathbb R^m)$ and $\zeta=\left(\zeta^1,\dots,\zeta^m\right)\in\mathfrak X(\mathcal B)\simeq C^\infty(\mathcal B,\mathbb R^m)$, which is given by the {\it standard inner product} on $\mathbb R^m$, will be denoted by $\omega\cdot\zeta=\sum_{i=1}^m\omega_i\zeta^i\in C^\infty(\mathcal B)$.

\begin{remark}[$x$-dependence]\rm The developments made in this paper remain valid in the more general case in which the Lagrangian density depends explicitly on $x \in \mathcal{B} $, i.e., we have $\mathfrak{L}: \mathcal{B}  \times \mathbb{R} \times \mathbb{R} \times \mathbb{R} ^m \rightarrow \mathbb{R} $, with associated Lagrangian
\begin{equation}\label{L_x}
L( \varphi , \nu )= \int_ \mathcal{B} \mathfrak{L}(x, \varphi (x), \nu (x), \nabla \varphi (x))dx,
\end{equation}
compare with \eqref{eq:lagrangianfunctions}. We refer to \S\ref{VM} for an example of this case.
\end{remark}

The following result ensures that the partial functional derivatives introduced above lie in the restricted dual, $V^\star\subset V'$, when the Lagrangian is defined through a density.

\begin{lemma}\label{lemma:partialderivativevector}\rm
Let $L:TV\to\mathbb R$ be a Lagrangian defined through a density, as in \eqref{eq:lagrangianfunctions}. Then the partial functional  derivatives of $L$ defined in \eqref{def_delta_L} lie in the resricted dual $V^\star \subset V^{\prime}$. Furthermore, under identification \eqref{eq:identificationrestricted}, they are given in terms of the partial derivatives of $\mathfrak{L}$ as
\begin{eqnarray*}
\frac{\delta L}{\delta\varphi}(\varphi,\dot\varphi) & = & \left(\frac{\partial\mathfrak L}{\partial\varphi}(\varphi,\dot\varphi,\nabla\varphi)-\operatorname{div}\frac{\partial\mathfrak L}{\partial\nabla\varphi}(\varphi,\dot\varphi,\nabla\varphi),~\left.\frac{\partial\mathfrak L}{\partial\nabla\varphi}(\varphi,\dot\varphi,\nabla\varphi)\right|_{\partial\mathcal B}\cdot n\right)\in V^\star,\\
\frac{\delta L}{\delta\nu}(\varphi,\dot\varphi) & = & \left(\frac{\partial\mathfrak L}{\partial\nu}(\varphi,\dot\varphi,\nabla\varphi),~0\right)\in V^\star,
\end{eqnarray*}
for each $(\varphi,\dot\varphi)\in TV$, where
\begin{equation*}
\operatorname{div}:\Omega^1(\mathcal B)\simeq C^\infty(\mathcal B,\mathbb R^m)\to C^\infty(\mathcal B),\quad\omega=\left(\omega^1,\dots,\omega^m\right)\mapsto\operatorname{div}\omega=\sum_{i=1}^m\frac{\partial\omega^i}{\partial x^i},
\end{equation*}
is the divergence, and $n\in C^\infty(\partial\mathcal B,\mathbb R^m)$ is the outward-pointing, unit, normal vector field on the boundary.
\end{lemma}

\begin{proof}
Let $\delta\varphi\in V$. By definition, we have
\begin{align*}
\frac{\delta L}{\delta\varphi}(\varphi,\dot\varphi)(\delta\varphi) & =\left.\frac{d}{d\epsilon}\right|_{\epsilon=0}\int_{\mathcal B}\mathfrak L(\varphi+\epsilon\,\delta\varphi,\dot\varphi,\nabla(\varphi+\epsilon\,\delta\varphi))\,{dx}\\[2mm]
& =\int_{\mathcal B}\left(\frac{\partial\mathfrak L}{\partial\varphi}(\varphi,\dot\varphi,\nabla\varphi)\,\delta\varphi+\frac{\partial\mathfrak L}{\partial\nabla\varphi}(\varphi,\dot\varphi,\nabla\varphi)\,\nabla(\delta\varphi)\right){dx}\\[2mm]
& =\int_{\mathcal B}\left(\frac{\partial\mathfrak L}{\partial\varphi}(\varphi,\dot\varphi,\nabla\varphi)-\operatorname{div}\frac{\partial\mathfrak L}{\partial\nabla\varphi}(\varphi,\dot\varphi,\nabla\varphi)\right)\,\delta\varphi\,{dx}\\
& \qquad   +\int_{\partial\mathcal B}\left(\frac{\partial\mathfrak L}{\partial\nabla\varphi}(\varphi,\dot\varphi,\nabla\varphi)\cdot n\right)\,\delta\varphi\,{ds},
\end{align*}
where we have used the standard integration by parts formula. Analogously, for $\delta\dot\varphi\in V$,
\begin{equation*}
\frac{\delta L}{\delta\nu}(\varphi,\dot\varphi)(\delta\dot\varphi)=\left.\frac{d}{d  \epsilon}\right|_{\epsilon=0}\int_{\mathcal B}\mathfrak L(\varphi,\dot\varphi+\epsilon\,\delta\dot\varphi,\nabla\varphi)\,{dx}=\int_{\mathcal B}\frac{\partial\mathfrak L}{\partial\nu}(\varphi,\dot\varphi,\nabla\varphi)\,\delta\dot\varphi\,{dx}.
\end{equation*}
Since the previous expressions hold for every $\delta\varphi, \delta\dot\varphi \in V$, the proof is completed.
\end{proof}

\medskip 

The differential of $L$ is the map defined as
\begin{equation*}
{\rm d}L:TV\to T'(TV),\quad(\varphi,\dot\varphi)\mapsto\left(\varphi,\dot\varphi,\frac{\delta L}{\delta\varphi}(\varphi,\dot\varphi),\frac{\delta L}{\delta\nu}(\varphi,\dot\varphi)\right).
\end{equation*}
Observe that the previous Lemma ensures that $ {\rm d} L$ takes values in $T^\star(TV)$. Hence, the \emph{Dirac differential} of $L$ can be defined as
\begin{equation}\label{eq:DLfunctions}
{\mathrm d}_DL=\gamma_{T^\star V}\circ{\rm d}  L:TV\to T^\star(T^\star V),\quad(\varphi,\dot\varphi)\mapsto\left(\varphi,\frac{\delta L}{\delta\nu}(\varphi,\dot\varphi),-\frac{\delta L}{\delta\varphi}(\varphi,\dot\varphi),\dot\varphi\right).
\end{equation}

Lastly, we consider body and boundary forces acting on the system, such as control forces. To that end, recall that the \emph{Legendre transform} of $L$ is given by
\begin{equation*}\label{LegendreTrans}
\mathbb FL:TV\to T^\star V,\quad(\varphi,\nu)\mapsto\left(\varphi,\frac{\delta L}{\delta\nu}(\varphi,\nu)\right)=\left(\varphi,\frac{\partial\mathfrak L}{\partial\nu}(\varphi,\nu,\nabla\varphi),0\right).
\end{equation*}

\begin{definition}\label{ExtForces}\rm
Let $F:TV\to T^\star V$ be an \emph{external force} with values in $T^\star V$, and write
\begin{equation}\label{bbf} F(\varphi,\nu)=\left(\varphi,\mathcal F(\varphi,\nu),\mathcal F_\partial(\varphi,\nu)\right),
\end{equation}
for each $(\varphi,\nu)\in TV$, where $\mathcal F:TV\to C^\infty(\mathcal B)$ is the external body force acting on the interior of $\mathcal B$ and $\mathcal F_\partial:TV\to C^\infty(\partial\mathcal B)$ is the external force acting on the boundary, $\partial \mathcal B$. For a given Lagrangian $L:TV \rightarrow \mathbb{R} $, the associated \textit{Lagrangian force field} is the map $\widetilde{F}: TV \to T^{\star}\left(T^{\star}V\right)$ defined as in \eqref{def_tilde_F} by 
\[
\left\langle \widetilde{F}(\varphi,\nu) , W\right\rangle = \left\langle F(\varphi,\nu), T_{\mathbb FL(\varphi,\nu)}\pi_{V}(W) \right\rangle,\qquad(\varphi,\nu) \in TV,~W\in T_{\mathbb FL(\varphi,\nu)}\left(T^\star V\right),
\] 
where $\pi_V:T^{\star}V \to V$ is the natural projection and $T\pi_V:T\left(T^\star V\right) \to TV$ denotes its tangent map. More explicitly, it reads
\begin{equation*}
\widetilde{F}(\varphi,\nu)= \left(\varphi,\frac{\partial\mathfrak L}{\partial\nu}(\varphi,\nu, \nabla \varphi),0,\mathcal F(\varphi,\nu),\mathcal F_\partial(\varphi,\nu),0\right).
\end{equation*}
\end{definition}

\begin{remark}[Functional and time dependence of the force]\rm In \eqref{bbf} there is no restriction on the way the forces $ \mathcal{F} $ and $ \mathcal{F} _ \partial $ depend on the fields $ \varphi $ and $ \nu $. A typical situation is a dependence on the point values of $ \varphi $, $ \nu $,  and $ \nabla \varphi $, i.e., $ \mathcal{F} ( \varphi, \nu )(x)=\mathfrak{F}(x, \varphi (x), \nu (x), \nabla \varphi (x))$ for all $x \in \mathcal{B} $. An explicit time dependence can be also considered, which is typical in control problems, see \S\ref{telegraph} for an example.
\end{remark}

We are ready to introduce Lagrange--Dirac mechanical systems with body and boundary forces.

\begin{definition}\label{def:LDfunctions}\rm
Consider a Lagrangian $L:TV \rightarrow \mathbb{R}$ and an exterior force $F:TV \rightarrow T^\star V$ given in terms of a Lagrangian density $\mathfrak{L}$, and in terms of interior and boundary forces $ \mathcal{F} $, and $ \mathcal{F} _ \partial $ as in \eqref{eq:lagrangianfunctions} and \eqref{bbf}. The associated \emph{Lagrange--Dirac dynamical system with body and boundary forces} for a curve $(\varphi,\nu,\alpha,\alpha_\partial):[t_0,t_1]\to TV\oplus T^\star V$ in the Pontryagin bundle is given by
\begin{equation}\label{LagDiracSys_ExtForces}
\big((\varphi,\alpha,\alpha_\partial, \dot\varphi,\dot\alpha,\dot\alpha_\partial), {\mathrm d}_D L(\varphi,\nu)-\widetilde{F}(\varphi,\nu)\big)\in D_{T^\star V} (\varphi,\alpha,\alpha_\partial),
\end{equation}
where $\dot{(\,)}=\partial_{t}(\,)$ denotes the time derivative.
\end{definition}

\begin{proposition}\label{prop:lagrangediracfunctions}\rm
A curve $
(\varphi,\nu,\alpha,\alpha_\partial):[t_0,t_1]\to TV\oplus T^\star V=C^\infty(\mathcal B)\times C^\infty(\mathcal B)\times C^\infty(\mathcal B)\times C^\infty(\partial\mathcal B)
$ is a solution of the Lagrange--Dirac system with body and boundary forces \eqref{LagDiracSys_ExtForces} if and only if it satisfies the following system of equations:
\begin{equation}\label{LagDiracSysEq_ExtForce}
\left\{
\begin{array}{ll}
\dot\varphi=\nu, & \\[4mm]
\displaystyle\alpha=\frac{\partial\mathfrak L}{\partial\nu}(\varphi,\nu,\nabla\varphi),\qquad & \displaystyle\dot\alpha=\frac{\partial\mathfrak L}{\partial\varphi}(\varphi,\nu,\nabla\varphi)-\operatorname{div}\frac{\partial\mathfrak L}{\partial\nabla\varphi}(\varphi,\nu,\nabla\varphi)+\mathcal F(\varphi,\nu),\vspace{0.2cm}\\
\alpha_\partial=0, & \displaystyle\dot\alpha_\partial=\left.\frac{\partial\mathfrak L}{\partial\nabla\varphi}(\varphi,\nu,\nabla\varphi)\right|_{\partial\mathcal B}\cdot n+\mathcal F_\partial(\varphi,\nu).
\end{array}\right.
\end{equation}
\end{proposition}

\begin{proof}
We get the desired results by direct computations from Equations \eqref{eq:DVfunctions} and \eqref{eq:DLfunctions}, as well as Definitions \ref{ExtForces} and \ref{def:LDfunctions}.
\end{proof}
\medskip 

Observe that the second and fourth conditions in \eqref{LagDiracSysEq_ExtForce}, which can be written as $(\varphi,\alpha,\alpha_\partial)=\mathbb FL(\varphi,\nu)$, arise from the condition that the base point of the covector  ${\mathrm d}_D L(\varphi,\nu)-\widetilde{F}(\varphi,\nu)$ must be equal to $( \varphi , \alpha , \alpha _ \partial )$ from  \eqref{LagDiracSys_ExtForces}.

\medskip 
Observe also that by eliminating the variables $\nu $, $ \alpha $ and $ \alpha _ \partial$ by using the first, second, and fourth equation in \eqref{LagDiracSysEq_ExtForce}, we get the equation in terms of $\varphi $ uniquely, as
\begin{equation}\label{EL_ExtForce}
\left\{
\begin{array}{l}
\displaystyle\vspace{0.2cm}\frac{d}{dt}  \frac{\partial \mathfrak L}{\partial \dot  \varphi }(\varphi, \dot\varphi ,\nabla\varphi)-\frac{\partial\mathfrak L}{\partial\varphi}(\varphi, \dot\varphi ,\nabla\varphi)+\operatorname{div}\frac{\partial\mathfrak L}{\partial\nabla\varphi}(\varphi, \dot\varphi ,\nabla\varphi)=\mathcal F(\varphi, \dot\varphi)\\
\displaystyle\left.\frac{\partial\mathfrak L}{\partial\nabla\varphi}(\varphi, \dot\varphi ,\nabla\varphi)\right|_{\partial\mathcal B}\cdot n=-\mathcal F_\partial(\varphi,\dot  \varphi ).
\end{array}
\right.
\end{equation} 
These are the \textit{Lagrange--d'Alembert equations with body and boundary exterior forces}.

\medskip 

Recall that the \emph{energy density associated to $\mathfrak{L}$} is given by
\begin{equation*}
\mathfrak{E}(\varphi,\nu,\nabla\varphi)=\frac{\partial\mathfrak{L}}{\partial\nu}(\varphi,\nu,\nabla\varphi)\,\nu-\mathfrak{L}(\varphi,\nu,\nabla\varphi),\qquad(\varphi,\nu,\nabla\varphi)\in\mathbb R\times\mathbb R\times\mathbb R^m.
\end{equation*}

With this definition, we directly get the energy balance equations, which follow from a direct computation using \eqref{LagDiracSysEq_ExtForce}.

\begin{proposition}[Energy balance]\rm 
If $(\varphi,\nu,\alpha,\alpha_\partial):[t_0,t_1]\to TV\oplus T^\star V$ is a solution of \eqref{LagDiracSys_ExtForces}, then the local and global energy balance equations are
\begin{align*}
\frac{\partial }{\partial t}\mathfrak{E}(\varphi,\nu,\nabla\varphi)
=-\operatorname{div}\left(\frac{\partial\mathfrak L}{\partial\nabla\varphi}(\varphi,\nu,\nabla\varphi) \nu \right) +\mathcal F(\varphi,\nu) \nu 
\end{align*}
and
\begin{equation}\label{GEB_1}
\frac{d}{dt} \int_ \mathcal{B} \mathfrak{E}(\varphi,\nu,\nabla\varphi)\,{dx}
=
\underbrace{\int_{\mathcal B}\mathcal F(\varphi,\nu)\,\nu\,{dx}}_{\text{spatially distributed contribution}}+\underbrace{\int_{\partial\mathcal B}\mathcal F_\partial(\varphi,\nu)\,\nu\,{ds}}_{\text{boundary contribution}}.
\end{equation}
\end{proposition}
This equation shows the explicit form of each contribution to the energy change, both within the entire domain and through its boundary.

The energy balance equation \eqref{GEB_1} can be regarded as a Lagrangian analogue of the energy balance equation for the distributed (infinite dimensional) port-Hamiltonian system with external variables; see equation (48) in \cite{VdSMa2002}, where the rate of total energy is balanced with the external power flow through the boundary $\partial\mathcal B$ and the distributed external power flow. Note that $\mathcal F_\partial(\varphi,\nu)(x)$ and $\nu(x)$ are respectively understood as the external effort and flow variables on the boundary $\partial\mathcal B$, while $\mathcal F(\varphi,\nu)(x)$ and $\nu(x)$ are the external effort and flow variables on $\mathcal B$. The corresponding energy balance for systems on $k$-forms will be given in \S\ref{sec:kforms}.

\begin{remark}\rm
Observe that for the unforced case, i.e., $F\equiv 0$, the previous result yields local and global energy conservation equations:
\begin{align*}
\frac{\partial }{\partial t}\mathfrak{E}(\varphi,\nu,\nabla\varphi)
+\operatorname{div}\left(\frac{\partial\mathfrak L}{\partial\nabla\varphi}(\varphi,\nu,\nabla\varphi) \nu \right) =0\qquad\text{and}\qquad \frac{d}{dt} \int_ \mathcal{B} \mathfrak{E}(\varphi,\nu,\nabla\varphi)\,{dx}
=0.
\end{align*}
\end{remark}

\subsection{Variational structures for infinite-dimensional Lagrange--Dirac dynamical systems}\label{variational_structures_1}

Recall that a variational structure associated to the Lagrange--Dirac dynamical system obtained in \eqref{LagDiracSysEq_ExtForce} must give the solution curve $( \varphi , \nu,  \alpha , \alpha _ \partial ):[t_0,t_1] \rightarrow TV \oplus T^\star V$ \textit{in the Pontryagin bundle} as a critical point condition. Before considering this, we give in \S\ref{LdA_BF} the variational structure associated to the equations \eqref{EL_ExtForce} for the solution curve $ \varphi : [t_0,t_1] \rightarrow V$ obtained after having eliminated the variables $ \nu , \alpha , \alpha _ \partial $. This is nothing else than an infinite-dimensional version of the Lagrange--d'Alembert principle for forced systems.

\subsubsection{Lagrange--d'Alembert principle for infinite-dimensional Lagrangian systems}\label{LdA_BF}

The general expression of the Lagrange--d'Alembert principle for a Lagrangian $L:TV \rightarrow \mathbb{R} $ and an exterior force $F:TV \rightarrow T^\star V$ is given by
\begin{equation}\label{LdA_general} 
\delta \int_{t_{0}}^{t_{1}} L(\varphi, \dot{\varphi})dt +\int_{t_{0}}^{t_{1}} \left<F(\varphi,\dot{\varphi}),\delta{\varphi}\right>dt=0,
\end{equation} 
for free variations $\delta\varphi$ vanishing at $t=t_0,t_1$.
In our case $V=C^{\infty}(\mathcal B)$, $L$ and $F$ are given as in \eqref{eq:lagrangianfunctions} and \eqref{bbf}, so we get the following result.

\begin{proposition}\rm
A curve $\varphi: [t_{0},t_{1}] \to C^\infty(\mathcal B)$ satisfies the critical condition for the \emph{Lagrange--d'Alembert action functional}, i.e.,
\begin{equation}\label{LdA_bbf}
\delta \int_{t_{0}}^{t_{1}}\left( \int_{\mathcal B}\mathfrak L(\varphi,\dot\varphi,\nabla\varphi)\,{dx} \right)dt+\int_{t_{0}}^{t_{1}}\left( \int_{\partial\mathcal B}\mathcal F_\partial(\varphi,\dot{\varphi})\,\delta{\varphi}\,{ds}+\int_{\mathcal B}\mathcal F(\varphi,\dot{\varphi})\,\delta{\varphi}\,{dx}\right) dt=0,
\end{equation}
for free variations $\delta\varphi$ vanishing at $t=t_0,t_1$, if and only if it satisfies the \emph{Lagrange--d'Alembert equations} given in \eqref{EL_ExtForce}. 

\end{proposition}
\begin{proof}
First we note that \eqref{LdA_bbf} follows from \eqref{LdA_general} by using \eqref{eq:lagrangianfunctions} and \eqref{bbf}. By taking the variations in \eqref{LdA_bbf}, we get 
\begin{equation*}
\begin{split}
&\int_{t_{0}}^{t_{1}}\int_{\mathcal B}\left( \frac{\partial \mathfrak L}{\partial \varphi} \delta\varphi +\frac{\partial \mathfrak L}{\partial\nu} \delta\dot\varphi+\frac{\partial \mathfrak L}{\partial \nabla\varphi}\delta\nabla\varphi\right)\,{dx}\, dt+\int_{t_{0}}^{t_{1}}\left( \int_{\partial\mathcal B}\mathcal F_\partial\,\delta{\varphi}\,{ds}+\int_{\mathcal B}\mathcal F\,\delta{\varphi}\,{dx}\right) dt\\[2mm]
&=\int_{t_{0}}^{t_{1}}\left[\int_{\mathcal B}\left(\frac{\partial \mathfrak L}{\partial \varphi}- \frac{\partial }{\partial t} \frac{\partial \mathfrak L}{\partial\nu}
-\operatorname{div}\frac{\partial\mathfrak L}{\partial\nabla\varphi} +\mathcal F\right)\delta\varphi \,{dx}\right]dt\\[2mm]
&\quad  +\int_{t_0}^{t_1}\left[\int_{\partial\mathcal B}\left(\left.\frac{\partial\mathfrak L}{\partial\nabla\varphi}\right|_{\partial\mathcal B}\cdot n+\mathcal F_\partial\right)\delta\varphi\,{ds} \right]dt+\int_{\mathcal B} \frac{\partial \mathfrak L}{\partial\nu}\delta \varphi \,{dx} \biggr{\arrowvert}_{t_{0}}^{t_{1}}\ =\ 0.
\end{split}
\end{equation*}
Since the variations $\delta\varphi$ are free and vanish at $t=t_0, t_1$, we get the result.
\end{proof}

\begin{remark}\rm
For the unforced case, the Lagrange--d'Alembert principle reduces to the Hamilton principle, and the Lagrange--d'Alembert equations yield the Euler--Lagrange equations.
\end{remark}

\subsubsection{Lagrange--d'Alembert--Pontryagin principle for infinite-dimensional Lagrange--Dirac dynamical systems}
Now, let us consider the variational principle for the solution curves of the infinite-dimensional Lagrange--Dirac systems on the Pontryagin bundle $TV \oplus T^\star V$ given in equation \eqref{LagDiracSys_ExtForces}, i.e., the Lagrange--d'Alembert--Pontryagin principle. Its general expression for a Lagrangian $L:TV \rightarrow \mathbb{R} $ and exterior force $F:TV \rightarrow T^\star V$, taking into account that $V=C^\infty( \mathcal{B} )$ and $V^\star= C^\infty( \mathcal{B} ) \times  C^\infty( \partial \mathcal{B} )$, is given as
\begin{equation}\label{LdAP2}
\delta\int_{t_{0}}^{t_{1}}\Bigl[ L(\varphi,\nu)+ \left< (\alpha, \alpha_{\partial}), (\dot\varphi-\nu)\right>\Bigr]dt+\int_{t_{0}}^{t_{1}} \left<F(\varphi,\dot{\varphi}),\delta{\varphi}\right> dt=0,
\end{equation}
for free variations $\delta\varphi,\delta \nu,\delta\alpha,\delta\alpha_\partial$ with $\delta\varphi$ vanishing at $t=t_0,t_1$. With $L$ and $F$ given as in \eqref{eq:lagrangianfunctions} and \eqref{bbf}, we get the following result.

\begin{proposition}\label{LdAPPrinciple_FamilyCinftyM}\rm
A curve $(\varphi,\nu,\alpha,\alpha_\partial):[t_0,t_1]\to TV\oplus T^\star V$ satisfies the critical condition for the {\it Lagrange--d'Alembert--Pontryagin action functional}, i.e., 
\begin{equation}\label{eq:LdAPcriticalfunctions}
\begin{split}
&\delta\int_{t_{0}}^{t_{1}}\left(\int_{\mathcal B}\mathfrak L(\varphi,\nu,\nabla\varphi)\,{dx}
+\int_{\mathcal B} \alpha (\dot\varphi-\nu) {dx} + \int_{\partial \mathcal B} \alpha_{\partial} (\dot\varphi-\nu)ds\right) dt\\[2mm]
&\hspace{5cm}
+\int_{t_{0}}^{t_{1}}\left(\int_{\partial\mathcal B}\mathcal F_\partial(\varphi,\dot{\varphi})\,\delta{\varphi}\,{ds}+\int_{\mathcal B}\mathcal F(\varphi,\dot{\varphi})\,\delta{\varphi}\,{dx}\right) dt\, = \, 0,
\end{split}
\end{equation}
for free variations $\delta\varphi,\delta \nu,\delta\alpha,\delta\alpha_\partial$ with $\delta\varphi$ vanishing at $t=t_0,t_1$, if and only if it satisfies the \emph{Lagrange--d'Alembert--Pontryagin equations}, which are exactly the equations of motion given in \eqref{LagDiracSysEq_ExtForce}.
\end{proposition}
\begin{proof} First we note that \eqref{eq:LdAPcriticalfunctions} follows from \eqref{LdAP2} by using \eqref{eq:lagrangianfunctions} and \eqref{bbf}, as well as
\[
\left< (\alpha, \alpha_{\partial}), (\dot\varphi-\nu)\right>= \int_{\mathcal B} \alpha (\dot\varphi-\nu) {dx} + \int_{\partial \mathcal B} \alpha_{\partial} (\dot\varphi-\nu)ds.
\]
The variation of the Lagrange--d'Alembert--Pontryagin action functional \eqref{eq:LdAPcriticalfunctions} yields
\begin{equation*}\label{HPVarActIntegral}
\begin{split}
&\int_{t_{0}}^{t_{1}}\bigg[\int_{\mathcal B}\left(\frac{\partial\mathfrak L}{\partial\varphi}(\varphi,\nu,\nabla\varphi)-\operatorname{div}\frac{\partial\mathfrak L}{\partial\nabla\varphi}(\varphi,\nu,\nabla\varphi)\right)\delta\varphi\,{dx}+\int_{\partial\mathcal B}\left(\left.\frac{\partial\mathfrak L}{\partial\nabla\varphi}\right|_{\partial\mathcal B}\cdot n\right)\delta\varphi\,ds\\[3mm]
&\hspace{2cm}+
\int_{\mathcal B}\frac{\partial\mathfrak L}{\partial\nu}(\varphi,\nu,\nabla\varphi)\delta\nu\,{dx}+\left.\int_{\mathcal B} \delta\alpha (\dot{\varphi}-\nu) {dx} + \int_{\partial \mathcal B} \delta\alpha_{\partial} (\dot{\varphi}-\nu)ds
\right.\\[3mm]
&\hspace{2cm}+\int_{\mathcal B} \alpha (\delta\dot{\varphi}-\delta \nu) {dx} + \int_{\partial \mathcal B} \alpha_{\partial} (\delta\dot{\varphi}-\delta \nu)\,ds
\bigg]dt\\[3mm]
&\hspace{1cm}+\int_{t_{0}}^{t_{1}}\bigg[\int_{\partial\mathcal B}\mathcal F_\partial(\varphi,\dot{\varphi})\,\delta{\varphi}\,{ds}+\int_{\mathcal B}\mathcal F(\varphi,\dot{\varphi})\,\delta{\varphi}\,{dx}\bigg] dt\\[3mm]
&=\int_{t_{0}}^{t_{1}}\bigg[\int_{\mathcal B}\left(-\dot{\alpha}+\frac{\partial\mathfrak L}{\partial\varphi}(\varphi,\nu,\nabla\varphi)-\operatorname{div}\frac{\partial\mathfrak L}{\partial\nabla\varphi}(\varphi,\nu,\nabla\varphi)+\mathcal F(\varphi,\dot\varphi)\right)\delta\varphi\,{dx}
\\[3mm]
&\hspace{1cm}
\left.+\int_{\partial\mathcal B}\left(-\dot\alpha_{\partial} +\left.\frac{\partial\mathfrak L}{\partial\nabla\varphi}\right|_{\partial\mathcal B}\cdot n+\mathcal F_\partial(\varphi,\dot\varphi)\right)\delta\varphi\,ds
+
\int_{\mathcal B}\left(-\alpha + \frac{\partial\mathfrak L}{\partial\nu}(\varphi,\nu,\nabla\varphi)\right)\delta \nu\,{dx}
\right.\\[3mm]
&\hspace{1cm}+\int_{\mathcal B} \delta\alpha (\dot{\varphi}-\nu) {dx} + \int_{\partial \mathcal B} \delta\alpha_{\partial} (\dot{\varphi}-\nu)ds- \int_{\partial \mathcal B} \alpha_{\partial} \delta \nu ds
\bigg]dt.
\end{split}
\end{equation*}
Therefore, from the critical condition \eqref{eq:LdAPcriticalfunctions}, we get the Lagrange--d'Alembert--Pontryagin equations of motion, which are nothing but the system equations for the forced Lagrange--Dirac dynamical system obtained in equation \eqref{LagDiracSysEq_ExtForce}.
\end{proof}

\begin{remark}\rm
For unforced systems, the Lagrange--d'Alembert--Pontryagin principle reduces to the Hamilton--Pontryagin principle, and the Lagrange--d'Alembert--Pontryagin equations yield the Hamilton--Pontryagin equations.
\end{remark}

Lastly, analogously to the finite-dimensional case (recall Theorem \ref{Theorem:LDfinitedim}), the results for infinite-dimensional Lagrange--Dirac systems are summarized in the following theorem.

\begin{theorem}\rm\label{Theorem:LDfunctions}
Consider the curves $(\varphi,\nu,\alpha,\alpha_\partial):[t_0,t_1]\to TV\oplus T^\star V$ and $\varphi=\rho_V\circ(\varphi,\nu,\alpha,\alpha_\partial): [t_{0},t_{1}] \to V=C^\infty(\mathcal B)$, where $\rho_V:TV\oplus T^\star V\to V$ is the natural projection. The following statements are equivalent:
\begin{itemize}
\item[(i)] The curve $\varphi: [t_{0},t_{1}] \to V$ is critical for the Lagrange--d'Alembert principle; namely, it satisfies
\begin{equation*}
\delta\int_{t_{0}}^{t_{1}} L(\varphi, \dot{\varphi})dt+\int_{t_{0}}^{t_{1}} \left<F(\varphi,\dot{\varphi}),\delta{\varphi}\right>dt=0,
\end{equation*}
for free variations $\delta\varphi$ vanishing at $t=t_0, t_1$ (see \eqref{LdA_bbf} for the principle written in terms of $\mathfrak{L}$, $ \mathcal{F} $, and $ \mathcal{F} _ \partial $).
\vskip 3pt
\item[(ii)] The curve $\varphi: [t_{0},t_{1}] \to V$ is a solution of the Lagrange--d'Alembert equations: 
\begin{equation*}
\left\{\begin{array}{ll}
\displaystyle\frac{\partial }{\partial t}\frac{\partial\mathfrak L}{\partial\nu}(\varphi, \dot\varphi,\nabla\varphi)=\frac{\partial\mathfrak L}{\partial\varphi}(\varphi, \dot\varphi,\nabla\varphi)-\operatorname{div}\frac{\partial\mathfrak L}{\partial\nabla\varphi}(\varphi, \dot\varphi,\nabla\varphi)+\mathcal F(\varphi,\dot{\varphi}),\\[5mm]
\displaystyle\mathcal F_\partial(\varphi, \dot\varphi)=-
\left.\frac{\partial\mathfrak L}{\partial\nabla\varphi}(\varphi,\dot\varphi,\nabla\varphi)\right|_{\partial\mathcal B}\cdot n.
\end{array}\right.
\end{equation*}
\vskip 3pt
\item[(iii)] The curve $(\varphi,\nu,\alpha,\alpha_\partial):[t_0,t_1]\to TV\oplus T^\star V$ is critical for the Lagrange--d'Alembert--Pontryagin principle; namely, it satisfies
\begin{equation*}
\begin{split}
&\delta\int_{t_{0}}^{t_{1}}\bigl[ L(\varphi, \nu)+ \left< (\alpha, \alpha_{\partial}), (\dot{\varphi}-\nu)\right>\bigr]dt+\int_{t_{0}}^{t_{1}} \left<F(\varphi,\nu),\delta{\varphi}\right>dt=0,
\end{split}
\end{equation*}
for free variations $\delta\varphi,\delta \nu,\delta\alpha,\delta\alpha_\partial$ with $\delta\varphi$ vanishing at $t=t_0,t_1$ (see \eqref{eq:LdAPcriticalfunctions} for the principle written in terms of $\mathfrak{L}$, $ \mathcal{F} $, and $ \mathcal{F} _ \partial $).
\vskip 3pt
\item[(iv)]  The curve $(\varphi,\nu,\alpha,\alpha_\partial):[t_0,t_1]\to TV\oplus T^\star V$ is a solution of the Lagrange--d'Alembert--Pontryagin equations of motion:
\begin{equation*}
\left\{\begin{array}{ll}
\dot\varphi=\nu,\vspace{0.2cm}\\
\displaystyle\alpha=\frac{\partial\mathfrak L}{\partial\nu}(\varphi,\nu,\nabla\varphi),\qquad & \displaystyle\dot\alpha=\frac{\partial\mathfrak L}{\partial\varphi}(\varphi,\nu,\nabla\varphi)-\operatorname{div}\frac{\partial\mathfrak L}{\partial\nabla\varphi}(\varphi,\nu,\nabla\varphi)+\mathcal F(\varphi,\nu),\vspace{0.2cm}\\
\alpha_\partial=0, & \displaystyle\dot\alpha_\partial=
\left.\frac{\partial\mathfrak L}{\partial\nabla\varphi}(\varphi,\nu,\nabla\varphi)\right|_{\partial\mathcal B}\cdot n+\mathcal F_\partial(\varphi,\nu).
\end{array}\right.
\end{equation*}
\vskip 3pt
\item[(v)] The curve $(\varphi,\nu,\alpha,\alpha_\partial):[t_0,t_1]\to TV\oplus T^\star V$ is a solution of the Lagrange--Dirac dynamical system given by
\begin{equation*}
\big((\varphi,\alpha,\alpha_\partial,\,\dot\varphi,\dot\alpha,\dot\alpha_\partial),{\mathrm d}_D L(\varphi,\nu)- \widetilde{F}(\varphi,\nu)\big)\in D_{T^\star V}(\varphi,\alpha,\alpha_\partial).
\end{equation*}
\end{itemize}
\end{theorem}

While (i) and (ii) are statements about a curve $ \varphi :[t_0, t_1] \rightarrow V$ in the configuration space of the system, (iii)--(v) are statements about a curve $( \varphi , \nu , \alpha , \alpha _ \partial ):[t_0,t_1] \rightarrow TV \oplus T^\star V$
in the restricted Pontryagin bundle of the configuration space. The equivalence of all these statements does not need the Lagrangian to be nondegenerate.

\subsection{Examples}\label{sec:examplesfunctions}
\subsubsection{Vibrating membrane and nonlinear wave equations}\label{VM}
The Lagrangian of a vibrating membrane (or higher dimensional analogs) with domain $\mathcal{B}$ is given by \eqref{eq:lagrangianfunctions} with the following Lagrangian density,
\begin{equation*}
\mathfrak L(\varphi,\nu,\nabla\varphi)=\frac{1}{2}\rho_0 \nu^2-\frac{1}{2}\tau|\nabla\varphi|^2,\qquad(\varphi,\nu,\nabla\varphi)\in\mathbb R\times\mathbb R\times\mathbb R^m,
\end{equation*}
where $\rho_0\in\mathbb R^+$ is the \emph{density} and $\tau\in\mathbb R^+$ is the \emph{tension}. Since the partial derivatives of $\mathfrak L$ are given by
\begin{equation*}
\frac{\partial\mathfrak L}{\partial\varphi}(\varphi,\nu,\nabla\varphi)=0,\qquad\frac{\partial\mathfrak L}{\partial\nu}(\varphi,\nu,\nabla\varphi)=\rho_0\nu,\qquad\frac{\partial\mathfrak L}{\partial\nabla\varphi}(\varphi,\nu,\nabla\varphi)=-\tau\nabla\varphi,
\end{equation*}
for each $(\varphi,\nu,\nabla\varphi)\in\mathbb R\times\mathbb R\times\mathbb R^m$, then the equations of motion for the Lagrange--Dirac dynamical system given in Proposition \ref{prop:lagrangediracfunctions} for a curve $(\varphi,\nu,\alpha,\alpha_\partial):[t_0,t_1]\to TV\oplus T^\star V$ read
\begin{equation*}
\left\{\begin{array}{ll}
\displaystyle\dot\varphi=\nu, & \vspace{0.1cm}\\
\displaystyle\alpha=\rho_0\nu,\qquad & \dot\alpha=\operatorname{div}(\tau\nabla\varphi)+ \mathcal{F} ( \varphi , \nu ) ,\vspace{0.1cm}\\
\alpha_\partial=0, & \dot\alpha_\partial=\left.-\tau\nabla\varphi\right|_{\partial\mathcal B}\cdot n + \mathcal{F} _ \partial ( \varphi , \nu  ),
\end{array}\right.
\end{equation*}
where we have considered the body and boundary forces $ \mathcal{F} $ and $ \mathcal{F} _ \partial $.
Observe that we obtain the \emph{forced wave equation}, $ \rho  _0\ddot\varphi=\tau\nabla^2\varphi + \mathcal{F} ( \varphi , \dot  \varphi )$, where $\nabla^2$ denotes the Laplacian, together with the \emph{nonhomogeneous} Neumann boundary condition, $ \tau \nabla\varphi|_{\partial\mathcal B}\cdot n= \mathcal{F} _ \partial ( \varphi , \nu )$.

\medskip

In the case where the density and tension depend on $x$, we are in the more general situation of a Lagrangian density depending also explicitly on $x$, see \eqref{L_x}. Our approach extends easily to this case, giving $ \rho  _0\ddot\varphi= \operatorname{div}( \tau\nabla\varphi) + \mathcal{F} ( \varphi , \dot  \varphi )$.

\medskip 

Note that we are using a Dirac structure on the phase space $T^\star V$ induced by the canonical symplectic form $ \Omega _{T^\star V}$, following the standard Hamiltonian formulation of the wave equations, e.g., \cite{ChMa1974}. This differs from the approach taken in the port-Hamiltonian literature, where this type of equation is treated by using energy variables appropriately augmented to include port variables. For instance, in the context of the wave equation, the port-Hamiltonian formulation is based on a constant Dirac structure $D \subset \mathcal{F} \oplus \mathcal{E}$ with $ \mathcal{F} = \mathcal{E} = C^\infty( \mathcal{B} , \mathbb{R} ^n ) \times C^\infty( \mathcal{B} , \mathbb{R} ) \times C^\infty ( \mathcal{B}, \mathbb{R}) \times C^\infty( \partial \mathcal{B}, \mathbb{R})$, see, e.g., \cite{HaMaSe2022}. In this approach, one needs to formulate the dynamics in terms of the energy variables $ \beta  \in  C^\infty( \mathcal{B} , \mathbb{R} ^n )$ and $ \alpha \in C^\infty( \mathcal{B}, \mathbb{R})$, subsequently chosen as $ \beta = \nabla \varphi $ and $ \alpha = \rho   _0 \nu $, in terms of which the Hamiltonian density must be expressible. This makes the approach harder to extend to general Lagrangian densities, such as those of nonlinear wave equations. 
Our approach, based on geometric state variables, is however directly applicable to these cases as well as to general expressions of the Lagrangian density $\mathfrak{L}( \varphi , \nu , \nabla \varphi )$. This is relevant for the treatment of nonlinear wave equations, such as those associated with Lagrangian densities of the form
\[
\mathfrak{L}( \varphi , \nu , \nabla \varphi )=\frac{1}{2}\rho_0 \nu^2-\frac{1}{2}\tau|\nabla\varphi|^2 -U ( \varphi ), 
\]
with $U( \varphi)$ representing potential energy. Typical expressions for $U$ are given as
\[
U( \varphi )= \frac{1}{2} m ^2\varphi ^2 + \frac{1}{p+1} \lambda \varphi ^{p+1}, \qquad p\geq 2,
\]
appearing, for example, in Klein-Gordon theories or
\[
U( \varphi)= 1-\cos \varphi,
\]
for the sine-Gordon case.
The Lagrange--Dirac dynamical system gives
\begin{equation*}
\left\{\begin{array}{ll}
\displaystyle\dot\varphi=\nu, & \vspace{0.1cm}\\
\displaystyle\alpha=\rho_0\nu,\qquad & \dot\alpha=-U'(\varphi)+ \operatorname{div}(\tau\nabla\varphi)+ \mathcal{F} ( \varphi , \nu ) ,\vspace{0.1cm}\\
\alpha_\partial=0, & \dot\alpha_\partial=\left.-\tau\nabla\varphi\right|_{\partial\mathcal B}\cdot n + \mathcal{F} _ \partial ( \varphi , \nu  ),
\end{array}\right.
\end{equation*}
with $U'(\varphi)= m ^2\varphi + \lambda \varphi ^p$ or $U'(\varphi) = \sin \varphi$.
This yields the nonlinear wave equation $\rho  _0\ddot\varphi= \operatorname{div}( \tau\nabla\varphi) -U'(\varphi)+ \mathcal{F} ( \varphi , \dot  \varphi )$ with boundary equation $ \tau \nabla\varphi|_{\partial\mathcal B}\cdot n= \mathcal{F} _ \partial ( \varphi , \nu )$.

\medskip

Our approach distinguishes itself from the port-Hamiltonian literature not only by employing geometric variables but also by defining the Dirac structure directly on the canonical phase space 
$T^\star V$ rather than on an extended port space. This extends the canonical symplectic formulation of mechanics on phase space, while the choice of the dual space $V^\star$ is playing a key role in the infinite-dimensional setting for the treatment of boundary flow. Importantly, this approach does not preclude the possibility of interconnecting Lagrange--Dirac systems via their ports in both finite and infinite-dimensional settings. Such interconnections can be handled within either the variational or Dirac framework. 
Given an interconnection constraint between two systems with configuration spaces $V_1$, $V_2$, there is a systematic geometric construction that yields the Dirac structure of the interconnected system on $T^\star (V_1 \times V_2)$. This is derived from the Dirac structures on $T^\star V_i$ (for $i=1,2$) of each subsystem, and from an interconnection condition defined by a distribution $\Sigma\subset T(V_1 \times V_2)$.
We refer to \cite{JaYo2014,GBYo2023} for the finite dimensional case and to \cite{GBYo2024} for preliminary results in the infinite-dimensional case. We also refer to Remark \ref{rmk_SDS} for further comments regarding the port-Hamiltonian approach.

\subsubsection{Telegraph equation}\label{telegraph}
As another illustrative example of a system with the external boundary term, consider a uniform lossless one-dimensional transmission line. The configuration space is given by $\Omega^{1}(\mathcal B)$ with  $\mathcal B=[0,1] \subset \mathbb{R}$. However, since $\mathcal B$ has a global coordinate, $x$, by denoting $\varphi=\varphi(x)\,{dx} \in \Omega^{1}(\mathcal B)$ (\emph{charge density} one-form) and $\nu=\nu(x)\,{dx} \in T_{\varphi}\Omega^{1}(\mathcal B)=\Omega^1(\mathcal B)$ (\emph{current density} one-form), we may choose $V=C^\infty(\mathcal B)$ as the configuration space. This way, the Lagrangian for this system is given by \eqref{eq:lagrangianfunctions} with
\[
\mathfrak{L}(\varphi,\nu,\varphi')=\frac{1}{2}\ell\nu^{2}-\frac{1}{2c}\left(\varphi^{\prime}\right)^{2},\qquad\left(\varphi,\nu,\varphi'\right)\in\mathbb R\times\mathbb R\times\mathbb R,
\]
where $\varphi'$ denotes the spatial derivative, and $\ell$ and $c$ indicate the \emph{constant of the distributed inductance} and the \emph{capacitance}, respectively. We suppose here that the force depends explicitly on time as $F(t, \varphi , \nu )$. Since the restricted dual is $ \partial \mathcal{B} = \{0,1\}$, we have $V^\star = C^\infty( \mathcal{B} ) \times C^\infty( \partial \mathcal{B} )= C^\infty( \mathcal{B} ) \times \mathbb{R} ^2 $, so that the force reads $F(t,\varphi,\nu)=\left(\varphi, \mathcal F(t,\varphi,\nu),\mathcal F_\partial(t,\varphi,\nu)\right)$ with $ \mathcal{F} _{ \partial }( t,\varphi , \nu ) \in \mathbb{R} ^2 $.
Suppose that the \emph{external  voltage}, is given by
\[ 
\mathcal F(t,\varphi,\nu)=0, \quad \mathcal F_\partial(t,\varphi,\nu)=(0, \mathcal{V}(t) ), 
\]
where $\mathcal V\in C^\infty([t_0,t_1])$.

In this case, the Lagrange--Dirac dynamical system given in Proposition \ref{prop:lagrangediracfunctions} becomes
\begin{equation*}
\left\{\begin{array}{ll}
\displaystyle\dot\varphi=\nu, & \vspace{0.1cm}\\[1mm]
\displaystyle\alpha=\ell\nu,\qquad & \displaystyle\dot\alpha=\frac{1}{c} \varphi^{\prime\prime},\vspace{0.1cm}\\[1mm]
\alpha_\partial=0, & \displaystyle\dot\alpha_\partial=\left.-\frac{1}{c} \varphi^{\prime}\right|_{\partial\mathcal B}+\mathcal F_\partial,
\end{array}\right.
\end{equation*}
which yields the \emph{telegraph equation}, $\ell \,c\, \ddot{\varphi}=\varphi^{\prime\prime}$, together with the boundary conditions $\varphi'(t,0)=0$ and $\varphi^{\prime}(t,1)=c \mathcal{V} (t)$.

\section{Lagrange--Dirac dynamical systems on the space of $k$-forms}\label{sec:kforms}

The aim of this section is to extend the theory of infinite-dimensional Lagrange--Dirac dynamical systems on $V=C^{\infty}(\mathcal B)$ to the case in which the configuration space is given by the space of differential $k$-forms
\begin{equation*}
V=\Omega^k(M),
\end{equation*}
where $M$ is an $m$-dimensional, compact smooth manifold with boundary and $0\leq k\leq m$. As in \S\ref{sec:functions}, we endow $V$ with the Fr\'echet space structure.

\subsection{Restricted dual and restricted cotangent bundle}\label{sec:restricteddualkforms}

For the treatment of systems with boundary flow on $\Omega^k(M)$, it is convenient to introduce, as in \S \ref{sec:restricteddualfunctions}, the \emph{restricted dual} of $V=\Omega^k(M)$, which is a vector subspace of the topological dual, $V'$, by means of a duality pairing given through integration on the manifold and its boundary. We denote by $\iota_{\partial M}: \partial M \rightarrow M$ the inclusion of the boundary.

\begin{definition}\rm\label{DualParing_FamilykForms}
The \emph{restricted dual space} of $V=\Omega^k(M)$ is defined as $V^\star=\Omega ^{m-k}(M) \times \Omega ^{m-k-1}( \partial M)$. The corresponding duality pairing is given by
\begin{equation}\label{duality_pairing_kforms}
\left\langle(\alpha,\alpha_\partial),\varphi\right\rangle=\int_{ M}\varphi\wedge\alpha+\int_{\partial M}\iota_{\partial M}^*\varphi\wedge\alpha_\partial,\qquad(\alpha,\alpha_\partial)\in V^\star,~\varphi\in V.
\end{equation}
\end{definition}

As in \eqref{eq:identificationrestricted}, this definition gives a continuous injection of $V^\star$ into the topological dual:
\begin{equation}\label{Psi} 
\Psi:V^\star \rightarrow V'.
\end{equation}

Of course, the pairing is weakly non-degenerate and, thus, the restricted cotangent bundle and the restricted iterated bundles may be defined in the same vein as in \S \ref{sec:restricteddualfunctions}. Briefly, the \emph{restricted cotangent bundle} is given by
\begin{equation*}
T^\star V=V\times V^\star=\Omega^k(M)\times\Omega^{m-k}(M)\times\Omega^{m-k-1}(\partial M),
\end{equation*}
and it is a subbundle of $T'V$ through the identification given by the corresponding pairing.
Similarly, the \emph{restricted iterated bundles} are given by
\begin{equation*}
\begin{array}{l}
T^\star(TV)=V\times V\times V^\star\times V^\star,\vspace{0.1cm}\\
T(T^\star V)=V\times V^\star\times V\times V^\star,\vspace{0.1cm}\\
T^\star(T^\star V)=V\times V^\star\times V^\star\times V.
\end{array}
\end{equation*}
To conclude, the \emph{restricted Pontryagin bundle} is given, as before, by
\begin{equation*}
T(T^\star V)\oplus T^\star(T^\star V)=V\times V^\star\times \left((V\times V^\star) \oplus (V^\star\times V)\right)\subset T(T^\star V)\oplus T'(T^\star V).
\end{equation*}

\subsection{Canonical forms, Tulczyjew triples and canonical Dirac structures}\label{sec:tulczyjewkforms}

The canonical symplectic forms and the Tulczyjew triple, together with the canonical Dirac structures, are introduced as in \S\ref{sec:tulczyjewfunctions}.

\begin{definition}\rm
The \emph{canonical one-form} on $T^\star V=\Omega^{m-k}(M) \times \Omega ^{m-k-1}( \partial M)$ is defined as
\[
\Theta_{T^\star V}(z) \cdot \delta{z}=\big< z, T_{z}\pi_{V}(\delta z)\big>,\qquad z \in T^\star V,~\delta z\in T_{z}(T^\star V),
\]
where $\pi_{V}: T^\star V \to V$ and the duality paring is given in \eqref{duality_pairing_kforms}. Furthermore, the \emph{canonical two-form} on $T^\star V$ is defined as $\Omega_{T^\star V}=-{\rm d}\Theta_{T^\star V}\in\Omega^2(T^\star \Omega^k(M))$. 
\end{definition}

More explicitly, we can write the canonical one-form as 
\begin{equation*}
\Theta_{T^\star V}(\varphi,\alpha,\alpha_\partial)\cdot(\delta\varphi,\delta\alpha,\delta\alpha_{\partial})=\big< (\varphi, \alpha,\alpha_\partial),\delta\varphi\big>=\int_{ M}\delta\varphi\wedge\alpha+\int_{\partial M}\iota_{\partial M}^*\delta\varphi\wedge\alpha_\partial,
\end{equation*}
for each $z=(\varphi,\alpha,\alpha_{\partial})\in T^\star V$, $\delta{z}=(\delta\varphi,\delta\alpha,\delta\alpha_{\partial}) \in T_{z}\left(T^\star V\right) \simeq V\times V^\star$. The canonical two-form is then found as
\begin{equation}\label{Omega_star}
\begin{aligned}
&\Omega_{T^\star V}(\varphi,\alpha,\alpha_\partial)((\dot\varphi,\dot\alpha,\dot\alpha_\partial),(\delta\varphi,\delta\alpha,\delta\alpha_\partial))\phantom{\int_M}\\
&=\big\langle(\delta\alpha,\delta\alpha_\partial),\dot\varphi\big\rangle-\big\langle(\dot\alpha,\dot\alpha_\partial),\delta\varphi\big\rangle\\
&=\int_M   \dot \varphi \wedge \delta\alpha- \delta\varphi\wedge \dot\alpha +\int_{\partial M} \iota_{\partial M}^*\dot \varphi \wedge \delta\alpha_\partial   - \iota_{\partial M}^*\delta\varphi\wedge \dot\alpha_\partial ,
\end{aligned}
\end{equation}
for each $(\varphi,\alpha,\alpha_\partial)\in T^\star V$ and $(\dot\varphi,\dot\alpha,\dot\alpha_\partial),(\delta\varphi,\delta\alpha,\delta\alpha_\partial)\in T_{(\varphi,\alpha,\alpha_\partial)}\left(T^\star V\right)\simeq V\times V^\star$. 

\begin{proposition}\rm
The flat map of the canonical symplectic form defines a vector bundle morphism over the identity; namely, under the identification given by the dual pairing $\langle\cdot,\cdot\rangle$, it reads
\begin{equation}\label{Omega_flat_kforms}
\Omega_{T^\star V}^\flat: T(T^\star V)\to T^\star(T^\star V),\quad(\varphi,\alpha,\alpha_\partial,\dot\varphi,\dot\alpha,\dot\alpha_\partial)\mapsto(\varphi,\alpha,\alpha_\partial, -\dot\alpha,-\dot\alpha_\partial,\dot\varphi).
\end{equation}
\end{proposition}

Naturally, Remark \ref{remark:weakvsstrong} keeps holding in this more general context. In particular $\Omega_{T^\star V}$ becomes a strong symplectic form, when we confine ourselves to restricted duals since \eqref{Omega_flat_kforms} is an isomorphism. 

\begin{definition}\label{def:Omegaflatkforms}\rm
Consider the isomorphism over the identity, $\operatorname{id}_V$, defined as
\begin{equation*}
\kappa_{T^\star V}:T(T^\star V)\to T^\star(TV),\quad(\varphi,\alpha,\alpha_\partial, \dot\varphi,\dot\alpha,\dot\alpha_\partial)\mapsto(\varphi,\dot\varphi, \dot\alpha,\dot\alpha_\partial,\alpha,\alpha_\partial).
\end{equation*}
The \emph{restricted Tulczyjew triple} on the space of $k$-forms is the structure of three isomorphisms between the restricted iterated bundles given by
\begin{equation*}
	\begin{tikzpicture}
			\matrix (m) [matrix of math nodes,row sep=0.1em,column sep=6em,minimum width=2em]
			{	T^\star(T\Omega^k(M)) & T(T^\star\Omega^k(M)) & T^\star(T^\star\Omega^k(M))\\
			\left(\varphi,\dot\varphi, \dot\alpha,\dot\alpha_\partial,\alpha,\alpha_\partial\right) & \left(\varphi,\alpha,\alpha_\partial, \dot\varphi,\dot\alpha,\dot\alpha_\partial\right) & \left(\varphi,\alpha,\alpha_\partial, -\dot\alpha,-\dot\alpha_\partial,\dot\varphi\right).\\};
			\path[-stealth]
			(m-1-1) edge [bend left = 25] node [above] {$\gamma_{T^\star\Omega^k(M)}=\Omega_{T^\star V}^\flat\circ\kappa_{T^\star V}^{-1}$} (m-1-3)
			(m-1-2) edge [] node [above] {$\kappa_{T^\star\Omega^k(M)}$} (m-1-1)
			(m-1-2) edge [] node [above] {$\Omega_{T^\star\Omega^k(M)}^\flat$} (m-1-3)
			(m-1-3)
			(m-2-2) edge [|->] node [] {} (m-2-3)
			(m-2-2) edge [|->] node [] {} (m-2-1);
	\end{tikzpicture}
\end{equation*}
\end{definition}

Next, we introduce the canonical Dirac structure.

\begin{definition}\rm\label{def:diracstructurekforms}
The \emph{canonical Dirac structure} on $T^\star V= \Omega^k(M)\times\Omega^{m-k}(M)\times\Omega^{m-k-1}(\partial M)$ is the subbundle $D_{T^\star V}=\operatorname{graph}\Omega_{T^\star V}^\flat$. For each $(\varphi,\alpha,\alpha_\partial)\in T^\star V$, it reads
\begin{equation}\label{canonical_D_star}
\begin{split}
D_{T^\star V}(\varphi,\alpha,\alpha_\partial)&=\big\{(\dot\varphi,\dot\alpha,\dot\alpha_\partial, \delta\alpha,\delta\alpha_\partial,\delta\varphi)\in T_{(\varphi,\alpha,\alpha_\partial)}(T^\star V)\oplus T_{(\varphi,\alpha,\alpha_\partial)}^\star(T^\star V)\mid\\
&\hspace{5cm}-\dot\alpha=\delta\alpha,~-\dot\alpha_\partial=\delta\alpha_\partial,~\dot\varphi=\delta\varphi\big\}.
\end{split}
\end{equation}
\end{definition}

\subsection{Partial derivatives of the Lagrangian density}\label{sec:partialderivativeskforms}

For systems with configuration space $V=\Omega^k(M)$, we shall consider Lagrangians $L$ of 
the form
\begin{equation}\label{eq:lagrangiankforms}
L:TV\to\mathbb R,\quad L(\varphi,\nu)=\int_{ M}\mathscr L(\varphi,\nu,{\rm d}\varphi),
\end{equation}
for a Lagrangian density
\begin{equation}\label{eq:lagrangiandensitykforms}
\mathscr L:\textstyle\bigwedge^k T^* M\times_M\bigwedge^kT^* M\times_M\bigwedge^{k+1} T^* M\to\bigwedge^m T^* M
\end{equation}
given as a vector bundle morphism over the identity, $\operatorname{id}_{ M}$. The partial derivatives of $\mathscr L$,
\begin{align*}
& \frac{\partial\mathscr L}{\partial\varphi}:\textstyle\bigwedge^k T^* M\times_M\bigwedge^kT^* M\times_M\bigwedge^{k+1} T^* M\to\bigwedge^{m-k} T^* M,\vspace{0.1cm}\nonumber\\
& \frac{\partial\mathscr L}{\partial\nu }:\textstyle\bigwedge^k T^* M\times_M\bigwedge^kT^* M\times_M\bigwedge^{k+1} T^* M\to\bigwedge^{m-k} T^* M,\vspace{0.1cm}\\
&\frac{\partial\mathscr L}{\partial\zeta}:\textstyle\bigwedge^k T^* M\times_M\bigwedge^kT^* M\times_M\bigwedge^{k+1} T^* M\to\bigwedge^{m-k-1} T^* M\nonumber
\end{align*}
are defined in the usual way. Namely, for each $x\in M$ and $(\varphi_x,\nu _x,\zeta_x)\in\bigwedge^k T_x^*M\times\bigwedge^k T_x^*M\times\bigwedge^{k+1}T_x^*M$, by
\begin{equation}\label{PD_2} \begin{aligned}
\delta\varphi_x\wedge\frac{\partial\mathscr L}{\partial\varphi}(\varphi_x,\nu _x,\zeta_x)=\left.\frac{d}{d\epsilon}\right|_{{\epsilon}=0}\mathscr L(\varphi_x+{\epsilon}\,\delta\varphi_x,\nu _x,\zeta_x),\qquad & \delta\varphi_x\in\textstyle\bigwedge^k T_x^*M,\\
\delta{\nu }_x\wedge\frac{\partial\mathscr L}{\partial\nu }(\varphi_x,\nu _x,\zeta_x)=\left.\frac{d}{d\epsilon}\right|_{\epsilon=0}\mathscr L(\varphi_x,\nu _x+{\epsilon}\,\delta\nu _x,\zeta_x),\qquad & \delta\nu _x\in\textstyle\bigwedge^k T_x^*M,\\
\delta{\zeta}_x\wedge\frac{\partial\mathscr L}{\partial\zeta}(\varphi_x,\nu _x,\zeta_x)=\left.\frac{d}{d\epsilon}\right|_{{\epsilon}=0}\mathscr L(\varphi_x,\nu _x,\zeta_x+{\epsilon}\,\delta\zeta_x),\qquad & \delta\zeta_x\in\textstyle\bigwedge^{k+1} T_x^*M.
\end{aligned}
\end{equation}

\subsection{Lagrange--Dirac dynamical systems}\label{sec:LDkforms}

The partial functional derivatives of a general Lagrangian function $L:TV \rightarrow\mathbb{R}$ are the maps
\begin{equation*}
\frac{\delta L}{\delta\varphi},~\frac{\delta L}{\delta\nu }:TV\to V',
\end{equation*}
defined as
\begin{equation*}
\frac{\delta L}{\delta\varphi}(\varphi,\nu)(\delta\varphi)=\left.\frac{d}{d \epsilon}\right|_{\epsilon=0}L(\varphi+\epsilon\,\delta\varphi,\nu),\qquad\frac{\delta L}{\delta\nu }(\varphi,\nu)(\delta\nu)=\left.\frac{d}{d\epsilon}\right|_{\epsilon=0}L(\varphi,\nu+\epsilon\,\delta\nu),
\end{equation*}
for each $(\varphi,\nu)\in TV$ and $\delta\varphi, \delta\nu\in V$.
The next result ensures that the functional derivatives of a Lagrangian of the form \eqref{eq:lagrangiankforms} can be regarded as elements of $V^\star\subset V'$ by means of the corresponding identification.

\begin{lemma}\label{lemma:partialderivativekforms}\rm
Let $L:TV\to\mathbb R$ be a Lagrangian defined through a density, as in \eqref{eq:lagrangiankforms}. Then, for each $(\varphi,\nu )\in TV$, the fiber derivatives of $L$ under the identification \eqref{Psi} read
\begin{eqnarray*}
\frac{\delta L}{\delta\varphi}(\varphi,\nu ) & = & \left(\frac{\partial\mathscr L}{\partial\varphi}(\varphi,\nu ,{\rm d}\varphi)- (-1)^k{\rm d}\frac{\partial\mathscr L}{\partial\zeta}(\varphi,\nu ,{\rm d}\varphi),~\iota_{\partial M}^*\left(\frac{\partial\mathscr L}{\partial\zeta}(\varphi,\nu ,{\rm d}\varphi)\right)\right)\in V^\star,\\
\frac{\delta L}{\delta\nu }(\varphi,\nu )& = & \left(\frac{\partial\mathscr L}{\partial\nu }(\varphi,\nu ,{\rm d}\varphi),~0\right)\in V^\star.
\end{eqnarray*} 
\end{lemma}

\begin{proof}
Let $\delta\varphi\in V$. A direct computation yields:
\begin{align*}
\frac{\delta L}{\delta\varphi}(\varphi,\nu )(\delta\varphi) & =\left.\frac{d}{d \epsilon}\right|_{\epsilon=0}\int_{ M}\mathscr L(\varphi+ \epsilon\,\delta\varphi,\nu ,{\rm d}(\varphi+\epsilon\,\delta\varphi))\\
& =\int_{ M}\left(\delta\varphi\wedge\frac{\partial\mathscr L}{\partial\varphi}(\varphi,\nu ,{\rm d}\varphi)+{\rm d}\delta\varphi\wedge\frac{\partial\mathscr L}{\partial\zeta}(\varphi,\nu ,{\rm d}\varphi)\right)\\
& =\int_{ M}\delta\varphi\wedge\left(\frac{\partial\mathscr L}{\partial\varphi}(\varphi,\nu ,{\rm d}\varphi)- (-1)^k{\rm d}\frac{\partial\mathscr L}{\partial\zeta}(\varphi,\nu ,{\rm d}\varphi)\right)\\
& \qquad +\int_{\partial M}\iota_{\partial M}^*\delta\varphi\wedge\iota_{\partial M}^*\left(\frac{\partial\mathscr L}{\partial\zeta}(\varphi,\nu ,{\rm d}\varphi)\right).
\end{align*}
For the last equality, we have used that
\begin{equation}\label{eq:exteriorderivative}
{\rm d}\left(\delta\varphi\wedge\chi\right)={\rm d}\delta\varphi\wedge\chi+(-1)^k\delta\varphi\wedge{\rm d}\chi,
\end{equation}
together with the Stokes theorem and the fact that $\iota_{\partial M}^*(\delta\varphi\wedge\chi)=\iota_{\partial M}^*\delta\varphi\wedge\iota_{\partial M}^*\chi$ for each $\chi\in\Omega^{m-k-1}(M)$. Analogously,
\begin{equation*}
\frac{\delta L}{\delta\nu }(\varphi,\nu )(\delta \nu )=\left.\frac{d}{d \epsilon}\right|_{ \epsilon=0}\int_{ M}\mathscr L(\varphi,\nu + \epsilon\,\delta \nu ,{\rm d}\varphi)=\int_{ M}\delta\nu\wedge\frac{\partial\mathscr L}{\partial\nu }(\varphi,\nu ,{\rm d}\varphi).
\end{equation*}
\end{proof}

The differential of a general Lagrangian function $L:TV \rightarrow\mathbb{R}$ is the map
\begin{equation*}
{\rm d}L:TV\to T'(TV),\quad(\varphi,\nu )\mapsto\left(\varphi,\nu ,\frac{\delta L}{\delta\varphi}(\varphi,\nu ),\frac{\delta L}{\delta\nu }(\varphi,\nu )\right).
\end{equation*}
As a consequence of the previous Lemma, it may be regarded as taking values in $T^\star(TV)$.  As in the previous section, the \emph{Dirac differential} of the Lagrangian \eqref{eq:lagrangiankforms} is defined as
\begin{equation}\label{eq:DiracDifferential_kForms}
{\rm d}_{ D} L=\gamma_{T^\star V}\circ{\rm d} L:TV\to T^\star(T^\star V),\quad(\varphi,\nu)\mapsto\left(\varphi,\frac{\delta L}{\delta\nu}(\varphi,\nu),-\frac{\delta L}{\delta\varphi}(\varphi,\nu),\nu\right).
\end{equation}
In order to introduce body and boundary forces, we consider the \emph{Legendre transform} of $L$. From Lemma \ref{lemma:partialderivativekforms} we obtain:
    \begin{equation*}
    \mathbb{F} L: TV \to T^{\star}V,\quad(\varphi,\nu)\mapsto\left(\varphi, \frac{\partial\mathscr L}{\partial\nu}(\varphi,\nu,{\rm d}\varphi), 0\right).
    \end{equation*}

\begin{definition}\rm\label{Forces_KForms}
Let $F:TV\to T^\star V$ be an \emph{external force with values in $T^\star V$}, and write \begin{equation}\label{Fstar} F(\varphi,\nu)=(\varphi,\mathcal F(\varphi,\nu),\mathcal F_\partial(\varphi,\nu)),
\end{equation}
for each $(\varphi,\nu) \in TV$.  As in Definition \ref{ExtForces}, given a Lagrangian $L$, the associated \textit{Lagrangian force field} is the map $\widetilde{F}: TV \to T^{\star}(T^\star V)$ defined as
\[
\left\langle \widetilde{F}(\varphi,\nu), W\right\rangle = \left\langle F(\varphi,\nu), T_{\mathbb FL(\varphi,\nu)}\pi_{V}(W) \right\rangle,\qquad (\varphi,\nu) \in TV,~W\in T_{\mathbb FL(\varphi,\nu)}(T^\star V),
\]
where $\pi_V:T^{\star}V \to V$ is the natural projection and $T\pi_V:T\left(T^\star V\right) \to TV$ denotes its tangent map. More explicitly, it reads
\begin{equation*}
\widetilde{F}(\varphi,\nu)= \left(\varphi,\frac{\partial\mathscr L}{\partial\nu}(\varphi,\nu,{\rm d}\varphi),0,\mathcal F(\varphi,\nu),\mathcal F_\partial(\varphi,\nu),0\right).
\end{equation*}
\end{definition}

\begin{remark}[Body and boundary forces]\rm Note that the expression of $F$ in \eqref{Fstar} contains a body and a boundary force. According to the identification of the restricted dual, these forces are given in the spaces $\mathcal{F}( \varphi , \nu ) \in \Omega^{m-k}(M)$ and $\mathcal{F} _ \partial ( \varphi , \nu ) \in \Omega^{m-k-1}(\partial M)$, respectively.
\end{remark}

We are ready to introduce Lagrange--Dirac mechanical systems with external body and boundary forces.

\begin{definition}\rm\label{LagDiracSys_Family_kForms}
Consider a Lagrangian $L:TV \rightarrow \mathbb{R}$, $V=\Omega^k(M)$, and an exterior force $F:TV \rightarrow T^\star V$ given in terms of a Lagrangian density $\mathfrak{L}$, and in terms of interior and boundary forces $ \mathcal{F} $, and $ \mathcal{F} _ \partial $ as in \eqref{eq:lagrangiankforms} and \eqref{Fstar}. The associated \emph{Lagrange--Dirac dynamical system with body and boundary forces} for a curve $(\varphi,\nu,\alpha,\alpha_\partial):[t_0,t_1]\to TV\oplus T^\star V$ in the Pontryagin bundle is given by
\begin{equation}\label{InfDimLagDiracSys_Family_kForms_Star}
\big((\varphi,\alpha,\alpha_\partial, \dot\varphi,\dot\alpha,\dot\alpha_\partial), {\mathrm d}_D L(\varphi,\nu)-\widetilde{F}(\varphi,\nu)\big)\in D_{T^\star V} (\varphi,\alpha,\alpha_\partial),
\end{equation}
where $\dot{(\,)}=\partial_{t}(\,)$ denotes the time derivative and $D_{T^\star V}$ is the Dirac structure given in \eqref{canonical_D_star}.
\end{definition}
\begin{proposition}\label{proposition:LDkforms}\rm
A curve $(\varphi,\nu, \alpha,\alpha_\partial): [t_0,t_1] \to TV \oplus T^\star V$ is a solution of \eqref{InfDimLagDiracSys_Family_kForms_Star} if and only if it satisfies the following system of equations:
\begin{equation}\label{LagDiracSystem_Family_kForms_Star}
\left\{\begin{array}{ll}
\dot\varphi=\nu, & \vspace{0.2cm}\\
\displaystyle\alpha=\frac{\partial\mathscr L}{\partial\nu}(\varphi,\nu,{\rm d}\varphi),\qquad & \displaystyle\dot\alpha=\frac{\partial\mathscr L}{\partial\varphi}(\varphi,\nu,{\rm d}\varphi)- (-1)^k{\rm d}\frac{\partial\mathscr L}{\partial\zeta}(\varphi,\nu,{\rm d}\varphi)+\mathcal F(\varphi,\nu),\vspace{0.2cm}\\
\alpha_\partial=0, & \displaystyle\dot\alpha_\partial=\iota_{\partial M}^*\left(\frac{\partial\mathscr L}{\partial\zeta}(\varphi,\nu,{\rm d}\varphi)\right)+\mathcal F_{\partial}(\varphi,\nu).
\end{array}\right.
\end{equation}
\end{proposition}

\begin{proof}
A direct computation using \eqref{eq:DiracDifferential_kForms}, as well as Definitions \ref{def:diracstructurekforms}, \ref{Forces_KForms} and \ref{LagDiracSys_Family_kForms} leads to the desired result.
\end{proof}

\begin{remark}[Constant external forces]\rm\label{remark:constantforces}
Let us consider the special case of constant exterior body and boundary forces, i.e., $\mathcal F(\varphi,\nu)=\mathcal F_0\in\Omega^{m-k}(M)$ and $\mathcal F_\partial(\varphi,\nu)=\mathcal F_{\partial,0}\in\Omega^{m-k-1}(\partial M)$ for each $(\varphi,\nu)\in TV$. In this case, they may be included directly in the Lagrangian; namely,
\[
L_F(\varphi ,\nu)=L(\varphi,\nu)+\langle(\mathcal F_0,\mathcal F_{\partial,0}),\varphi\rangle=\int_M\left(\mathscr{L}( \varphi , \nu , {\rm d} \varphi ) + \varphi\wedge\mathcal{F}_0 \right)+ \int_{ \partial M} \iota_{\partial M}^*\varphi\wedge\mathcal{F}_{\partial,0},
\]
for each $(\varphi,\nu)\in TV$. An easy check using \eqref{eq:DiracDifferential_kForms}, as well as Definition \ref{def:diracstructurekforms} and Lemma \ref{lemma:partialderivativekforms}, shows that the Lagrange--Dirac system \eqref{InfDimLagDiracSys_Family_kForms_Star}  without the forcing term $\widetilde{F}$ in it, but with $L$ replaced by $L_F$, yields equations \eqref{LagDiracSystem_Family_kForms_Star}.
\end{remark}

The \emph{energy density associated to $\mathscr{L}$} is defined by
\begin{equation*}
\mathscr{E}(\varphi_x,\nu_x,\zeta_x)=\nu_x\wedge\frac{\partial\mathscr{L}}{\partial\nu}(\varphi_x,\nu_x,\zeta_x)-\mathscr{L}(\varphi_x,\nu_x,\zeta_x),
\end{equation*}
for each $(\varphi_x,\nu_x,\zeta_x)\in\bigwedge^k T^*M\times_M\bigwedge^k T^*M\times_M\bigwedge^{k+1}T^*M$. The following result gives the local and global energy balance along the solutions of the forced Lagrange--Dirac equations.

\begin{proposition}[Energy balance]\label{prop:energykforms}\rm 
If $(\varphi,\nu,\alpha,\alpha_\partial):[t_0,t_1]\to TV\oplus T^\star V$ is a solution of \eqref{InfDimLagDiracSys_Family_kForms_Star}, then
\begin{align*}
\frac{\partial }{\partial t}\mathscr{E}(\varphi,\nu,{\rm d}\varphi)=-{\rm d}\left(\nu\wedge\frac{\partial\mathscr L}{\partial\zeta}(\varphi,\nu,{\rm d}\varphi)\right)+\nu\wedge\mathcal F(\varphi,\nu),
\end{align*}
and
\begin{equation}\label{GEB}
\frac{d}{dt} \int_M \mathscr{E}(\varphi,\nu,{\rm d}\varphi)=\underbrace{\int_M\nu\wedge\mathcal F(\varphi,\nu)}_{\text{spatially distributed contribution}}+\underbrace{\int_{\partial M}\iota_{\partial M}^*\nu\wedge\mathcal F_\partial(\varphi,\nu)}_{\text{boundary contribution}}.
\end{equation}

\end{proposition}

\begin{proof}
By taking the time derivative of $\mathscr{E}(\varphi,\nu,{\rm d}\varphi)$ along the solution curve and using the equations of motion \eqref{LagDiracSystem_Family_kForms_Star}, we obtain
\begin{equation*}
\begin{split}
\frac{\partial }{\partial t}\mathscr{E}(\varphi,\nu,{\rm d}\varphi)
&=\dot\nu\wedge\frac{\partial\mathscr{L}}{\partial\nu}+\nu\wedge\left(\frac{\partial }{\partial t}\frac{\partial\mathscr{L}}{\partial\nu}\right)-\left(\dot\varphi\wedge\frac{\partial\mathscr{L}}{\partial\varphi}+\dot\nu\wedge\frac{\partial\mathscr{L}}{\partial\nu}+{\rm d}\dot\varphi\wedge\frac{\partial\mathscr{L}}{\partial\zeta}\right)\\
&=\nu\wedge\left(-(-1)^k{\rm d}\frac{\partial\mathscr L}{\partial\zeta}+\mathcal F(\varphi,\nu)\right)-{\rm d}\nu\wedge\frac{\partial\mathscr{L}}{\partial\zeta}\\
&=-{\rm d}\left(\nu\wedge\frac{\partial\mathscr L}{\partial\zeta}(\varphi,\nu,{\rm d}\varphi) \right)+\nu\wedge\mathcal F(\varphi,\nu).
\end{split}
\end{equation*}
By using integration by parts, we arrive at the desired result:
\begin{equation*}
\begin{split}
\frac{d}{dt} \int_M \mathscr{E}(\varphi,\nu,{\rm d}\varphi)
&=\int_{M}\left[-{\rm d}\left(\nu\wedge\frac{\partial\mathscr L}{\partial\zeta}(\varphi,\nu,{\rm d}\varphi)\right)+\nu\wedge\mathcal F(\varphi,\nu)\right]\\
&=\int_{\partial M}-\iota_{\partial M}^*\nu\wedge\iota_{\partial M}^*\left(\frac{\partial\mathscr L}{\partial\zeta}\right)+\int_{M}\nu\wedge\mathcal F(\varphi,\nu)\\
&=\int_{\partial M}\iota_{\partial M}^*\nu\wedge\mathcal F_\partial(\varphi,\nu) +\int_{M}\nu\wedge\mathcal F(\varphi,\nu).
\end{split}
\end{equation*}
\end{proof}

\begin{remark}\rm As in \eqref{GEB_1}, this equation shows the explicit form of each contribution to the energy change, both within
the entire domain and through its boundary, now in the setting of $k$-forms.
Observe that for the unforced case, i.e., $F=0$, the previous result yields the energy conservation:
\begin{align*}
\frac{\partial }{\partial t}\mathscr{E}(\varphi,\nu,{\rm d}\varphi)+{\rm d}\left(\nu\wedge\frac{\partial\mathscr L}{\partial\zeta}(\varphi,\nu,{\rm d}\varphi)\right)=0,\qquad & \frac{d}{dt} \int_M \mathscr{E}(\varphi,\nu,{\rm d}\varphi)=0.
\end{align*}
\end{remark}

\subsection{Variational structures for Lagrange--Dirac dynamical systems on the space of \texorpdfstring{$k$}{k}-forms}

Before giving the variational structure associated to the Lagrange--Dirac dynamical system for the solution curve $( \varphi , \nu,  \alpha , \alpha _ \partial ):[t_0,t_1] \rightarrow TV \oplus T^\star V$ \textit{in the Pontryagin bundle}, we give the variational structure for the 
solution curve $ \varphi : [t_0,t_1] \rightarrow V$ obtained after having eliminated the variables $ \nu , \alpha , \alpha _ \partial $. This is an infinite-dimensional version of the Lagrange--d'Alembert principle for forced systems.

\subsubsection{Lagrange--d'Alembert principle for infinite-dimensional Lagrangian systems}

Given a Lagrangian $L:TV \rightarrow \mathbb{R} $ and a force $F:TV \rightarrow T^\star V$, the general form of the Lagrange--d'Alembert principle is
\[
\delta\int_{t_{0}}^{t_{1}} L(\varphi, \dot{\varphi})\,dt+\int_{t_{0}}^{t_{1}} \left<F(\varphi,\dot{\varphi}), \delta{\varphi}\right>dt=0.
\]
By specializing the Lagrangian and the force using \eqref{eq:lagrangiankforms} and \eqref{Fstar}, we get the \emph{Lagrange--d'Alembert principle} for the Lagrange--Dirac dynamical systems on the space of \texorpdfstring{$k$}{k}-forms as follows.

\begin{proposition}\label{LdAPrinLag_kForms}\rm
A curve $\varphi: [t_{0},t_{1}] \to\Omega^{k}(M)$ is critical for the Lagrange--d'Alembert action functional, i.e.,
\begin{equation*}
\delta\int_{t_{0}}^{t_{1}}\left( \int_{M}\mathscr L(\varphi,\dot\varphi,{\rm d}\varphi)\right)\,dt+\int_{t_{0}}^{t_{1}}\left( \int_{M} \delta{\varphi}\wedge\mathcal F (\varphi,\dot{\varphi})+\int_{\partial M}\iota_{\partial M}^*\delta{\varphi}\wedge\mathcal F _\partial(\varphi,\dot{\varphi})\right)dt=0,
\end{equation*}
for free variations $\delta\varphi$ vanishing at $t=t_0, t_1$, if and only if it satisfies the \emph{Lagrange--d'Alembert equations}:
\begin{equation*}
\left\{\begin{array}{l}
\displaystyle\frac{\partial }{\partial t}\frac{\partial\mathscr L}{\partial\nu}(\varphi,\dot\varphi,{\rm d}\varphi)=\frac{\partial\mathscr L}{\partial\varphi}(\varphi,\dot\varphi,{\rm d}\varphi)-(-1)^k\operatorname{\rm d}\frac{\partial\mathscr L}{\partial\zeta}(\varphi,\dot\varphi,{\rm d}\varphi)+\mathcal F (\varphi,\dot{\varphi}),\vspace{2mm}\\
\displaystyle\mathcal F _\partial(\varphi, \dot\varphi)=-\iota_{\partial M}^*\left(\frac{\partial\mathscr L}{\partial\zeta}(\varphi,\dot\varphi,{\rm d}\varphi)\right).
\end{array}\right.
\end{equation*}
\end{proposition}
\begin{proof} The critical condition is computed as follows:
\begin{equation*}
\begin{split}
0&=\int_{t_{0}}^{t_{1}}\int_{M}\left(\delta\varphi\wedge\frac{\partial \mathscr L}{\partial \varphi} +\delta\dot\varphi\wedge\frac{\partial \mathscr L}{\partial\nu}+{\rm d}\delta\varphi\wedge\frac{\partial \mathscr L}{\partial \zeta}\right) dt+\int_{t_0}^{t_1}\langle F(\varphi,\dot\varphi),\delta\varphi\rangle dt\\[2mm]
&=\int_{t_{0}}^{t_{1}}\left[\int_{M}\delta\varphi\wedge\left(\frac{\partial \mathscr L}{\partial \varphi}- \frac{\partial }{\partial t} \frac{\partial \mathscr L}{\partial\nu}
- (-1)^k{\rm d}\frac{\partial\mathscr L}{\partial\zeta}+\mathcal F(\varphi,\dot\varphi)\right)\right.\\[2mm]
&\hspace{1cm}\left.+\int_{\partial M}\iota_{\partial M}^*\delta\varphi\wedge\left(\iota_{\partial M}^*\left(\frac{\partial\mathscr L}{\partial\zeta}\right)+\mathcal F_\partial(\varphi,\dot\varphi)\right) \right]dt
+\left[\int_M\delta \varphi\wedge\frac{\partial \mathscr L}{\partial\nu}\right]_{t_{0}}^{t_{1}},
\end{split}
\end{equation*}
where we have used \eqref{eq:exteriorderivative} and the Stokes theorem. Since the variations $ \delta  \varphi $ are free and vanish at $t=t_0, t_1$, we conclude.
\end{proof}

\begin{remark}\rm
For the unforced case, the Lagrange--d'Alembert principle reduces to the Hamilton principle, which yields the Euler--Lagrange equations.
\end{remark}

\subsubsection{Lagrange--d'Alembert--Pontryagin principle for Lagrange--Dirac dynamical systems on the space of  \texorpdfstring{$k$}{k}-forms}

The Lagrange--d'Alembert--Pontryagin principle constitutes the natural extension of the Lagrange--d'Alembert principle given in Proposition \ref{LdAPrinLag_kForms} that corresponds to the dynamics as given by the Lagrange--Dirac system. Its general form reads
\[
\delta\int_{t_{0}}^{t_{1}}\bigl( L(\varphi,\nu)+ \left< (\alpha, \alpha_{\partial}), \dot\varphi-\nu\right>\bigr)\,dt+\int_{t_{0}}^{t_{1}} \left<F(\varphi,\dot{\varphi}), \delta{\varphi}\right> dt=0.
\]

\begin{proposition}\rm\label{HPPrinciple_OmegakM}
A curve $(\varphi,\nu,\alpha,\alpha_\partial):[t_0,t_1]\to TV\oplus T^\star V$ is critical for the \emph{Lagrange--d'Alembert--Pontryagin action functional}, i.e.,
\begin{equation*}\label{HPActInt_Star}
\begin{split}
&\delta\int_{t_{0}}^{t_{1}}\left(\int_{M}\mathscr L(\varphi,\nu,{\rm d}\varphi)
+\int_{M}(\dot\varphi-\nu)\wedge\alpha  + \int_{\partial M} (\dot\varphi-\nu)\wedge\alpha_{\partial}\right) dt\\[2mm]
&\hspace{3cm}
+\int_{t_{0}}^{t_{1}}\left(\int_{M}\delta{\varphi}\wedge\mathcal F(\varphi,\dot{\varphi})+ \int_{\partial M}\delta{\varphi}\wedge\mathcal F_\partial(\varphi,\dot{\varphi})\right) dt=0,
\end{split}
\end{equation*}
for free variations $\delta\varphi,\delta \nu,\delta\alpha,\delta\alpha_\partial$ with $\delta\varphi$ vanishing at $t=t_0,t_1$, if and only if the curve $(\varphi,\nu,\alpha,\alpha_\partial) \in TV\oplus T^\star V$ satisfies the \textit{Lagrange--d'Alembert--Pontryagin equations}, which are exactly the equations of motion given in \eqref{LagDiracSystem_Family_kForms_Star}.
\end{proposition}
\begin{proof}
The results are obtained by direct computations, as in the proof of Proposition \ref{LdAPrinLag_kForms}. \end{proof}

\begin{remark}\rm
For unforced systems, the Lagrange--d'Alembert--Pontryagin principle reduces to the Hamilton--Pontryagin principle, which yields the Hamilton--Pontryagin equations.
\end{remark}

Lastly, analogous to the previous cases (recall Theorems \ref{Theorem:LDfinitedim} and \ref{Theorem:LDfunctions}), the results of this section are summarized in the following theorem.

\begin{theorem}\label{theorem:ForcedLDkforms}\rm
Let $(\varphi,\nu, \alpha,\alpha_\partial): [t_0,t_1] \to TV \oplus T^\star V$ and $\varphi=\rho\circ(\varphi,\nu, \alpha,\alpha_\partial): [t_0,t_1] \to V$, where $\rho:TV\oplus T^\star V\to V$ is the natural projection. The following  statements are equivalent:
\vskip 2pt
\begin{itemize}
\item[(a)] The curve $\varphi: [t_0,t_1] \to V$ is critical for the Lagrange--d'Alembert action functional:
\begin{equation*}
\begin{split}
&\delta\int_{t_{0}}^{t_{1}} L(\varphi, \dot{\varphi})\,dt+\int_{t_{0}}^{t_{1}} \left<F(\varphi,\dot{\varphi}), \delta{\varphi}\right> dt=0,
\end{split}
\end{equation*}
for free variations $\delta\varphi$ vanishing at $t=t_0, t_1$.
\vskip 5pt
\item[(b)] The curve $\varphi: [t_0,t_1] \to V$ is a solution of the Lagrange--d'Alembert equations:
\begin{equation*}
\left\{\begin{array}{l}
\displaystyle\frac{\partial }{\partial t}\frac{\partial\mathscr L}{\partial\nu}(\varphi,\dot\varphi,{\rm d}\varphi)=\frac{\partial\mathscr L}{\partial\varphi}(\varphi,\dot\varphi,{\rm d}\varphi)-(-1)^k\operatorname{\rm d}\frac{\partial\mathscr L}{\partial\zeta}(\varphi,\dot\varphi,{\rm d}\varphi)+\mathcal F (\varphi,\dot{\varphi}),\vspace{2mm}\\
\displaystyle\mathcal F _\partial(\varphi, \dot\varphi)=-\iota_{\partial M}^*\left(\frac{\partial\mathscr L}{\partial\zeta}(\varphi,\dot\varphi,{\rm d}\varphi)\right).
\end{array}\right.
\end{equation*}
\vskip 5pt
\item[(c)] The curve $(\varphi,\nu, \alpha,\alpha_\partial): [t_0,t_1] \to TV \oplus T^\star V$ satisfies the critical condition for the Lagrange--d'Alembert--Pontryagin action functional:
\begin{equation*}
\begin{split}
&\delta\int_{t_{0}}^{t_{1}}\bigl( L(\varphi,\nu)+ \left< (\alpha, \alpha_{\partial}), (\dot\varphi-\nu)\right>\bigr)dt+\int_{t_{0}}^{t_{1}} \left<F(\varphi,\dot{\varphi}), \delta{\varphi}\right> dt=0,
\end{split}
\end{equation*}
for free variations $\delta\varphi,\delta \nu,\delta\alpha,\delta\alpha_\partial$ with $\delta\varphi$ vanishing at $t=t_0,t_1$.
\vskip 5pt
\item[(d)] The curve $(\varphi,\nu, \alpha,\alpha_\partial): [t_0,t_1] \to TV \oplus T^\star V$ is a solution of the Lagrange--d'Alembert--Pontryagin equations:
\begin{equation*}
\left\{\begin{array}{ll}
\dot\varphi=\nu,&\\[3mm]
\displaystyle\alpha=\frac{\partial\mathscr L}{\partial\nu}(\varphi,\nu,{\rm d}\varphi),\; & \displaystyle\dot\alpha=\frac{\partial\mathscr L}{\partial\varphi}(\varphi,\nu,{\rm d}\varphi)-(-1)^k{\rm d}\frac{\partial\mathscr L}{\partial\zeta}(\varphi,\nu,{\rm d}\varphi)+\mathcal F(\varphi,\nu),\vspace{0.2cm}\\
\alpha_\partial=0, & \displaystyle\dot\alpha_\partial=\iota_{\partial M}^*\left(\frac{\partial\mathscr L}{\partial\zeta}(\varphi,\nu,{\rm d}\varphi)\right)+\mathcal F_{\partial}(\varphi,\nu).
\end{array}\right.
\end{equation*}
\vskip 5pt
\item[(e)] The curve $(\varphi,\nu, \alpha,\alpha_\partial): [t_0,t_1] \to TV \oplus T^\star V$ is a solution of the Lagrange--Dirac dynamical system
$(V=\Omega^k(M), D_{T^\star V}, L,F)$, i.e., it satisfies the condition 
\[
\big((\varphi,\alpha,\alpha_\partial,\dot\varphi,\dot\alpha,\dot\alpha_\partial),{\rm d}_{ D} L(\varphi,\nu)-\widetilde{F}(\varphi,\nu)\big)\in D_{T^\star V}(\varphi,\alpha,\alpha_\partial).
\]
\end{itemize}
\end{theorem}

\begin{remark}\label{rmk_SDS}\rm[Link with Stokes--Dirac structures (\cite{VdSMa2002, VaYoLeMa2010, SeVanSch2012})]
In the context of distributed port-Hamiltonian systems, the notion of Stokes--Dirac structures was proposed by \cite{VdSMa2002} and has been widely utilized in the field of nonlinear control theory. Stokes--Dirac structures are a specific type of constant Dirac structures appropriate for the treatment of system in which the energy and port variables are described by differential forms. It was shown in \cite{VaYoLeMa2010,SeVanSch2012} that Stokes--Dirac structures can be understood in the context of reduction by gauge symmetries of a Dirac structure associated to a canonical constant bivector.
We now illustrate the key difference between the Stokes--Dirac structures and the canonical Dirac structure we are considering here.

Recall that in our setting, the canonical Dirac structure $D_{T^\star V}$ is given on the cotangent bundle $T^\star V=V \times V^\star$, i.e., it is a subbundle $D_{T^\star V}\subset T (T^\star  V) \oplus T^\star (T^\star V)$, with the duality paring between $V$ and $V^\star$ given by 
\begin{equation*}
\left\langle(\alpha,\alpha_\partial),\varphi\right\rangle=\int_{ M}\varphi\wedge\alpha+\int_{\partial M}\iota_{\partial M}^*\varphi\wedge\alpha_\partial,
\end{equation*}
for each $(\alpha,\alpha_\partial)\in V^\star=\Omega^{m-k}(M) \times \Omega^{m-k-1}(\partial M)$ and $\varphi\in V=\Omega^k(M)$. We note that the restricted dual space $V^\star$ naturally includes the boundary force $\alpha_\partial \in \Omega^{m-k-1}(M)$ in addition to the distributed (interior) force $\alpha \in \Omega^{m-k}(M)$. Recall also that the canonical Dirac structure is obtained as the graph of the \textit{canonical symplectic form} on $T^\star V$, defined by using the duality paring that incorporates the boundary term (see \eqref{Omega_star} and Definition \ref{def:diracstructurekforms}).

On the other hand, in the setting of Stokes--Dirac structures on $V=\Omega^k(M)$, the dual space is chosen to be the ordinary smooth dual $V^\ast=\Omega^{m-k}(M)$ (not $V^\star$). In addition, the space of boundary flows, $\mathsf{F}=\Omega^{m-k-1}(\partial M)$, and its dual space, the space of boundary efforts, $\mathsf{E}:=\mathsf{F}^{\ast}=\Omega^{k}(\partial M)$, are employed. Then, for each $(\rho, \pi, \rho_{b}) \in V \times V^{\ast} \times \mathsf{F}^{\ast} = T^{\ast}V \times \mathsf{F}^{\ast}$, by introducing the duality paring between $(\dot{\rho}, \dot{\pi}, \dot{\rho}_{b})\in V \times V^{\ast} \times \mathsf{F}^{\ast} \simeq T_{(\rho, \pi, \rho_{b})}(T^{\ast}V \times \mathsf{F}^{\ast})$ and $(e_{\rho},e_{\pi},e_{b})\in V ^{\ast}\times V  \times \mathsf{F}\simeq T^{\ast}_{(\rho, \pi, \rho_{b})}(T^{\ast}V \times \mathsf{F}^{\ast})$ given by 
\begin{equation*}
\left\langle(e_{\rho},e_{\pi},e_{b}), (\dot{\rho}, \dot{\pi}, \dot{\rho}_{b}) \right\rangle=\int_{ M}(e_{\rho}\wedge \dot{\rho} + e_{\pi} \wedge \dot{\pi}) +\int_{\partial M}(e_{b}\wedge \dot{\rho}_{b} + e_{b} \wedge \iota_{\partial M}^*\dot{\rho}),
\end{equation*}
the Dirac structure is defined by the graph of
\begin{equation*}
\begin{array}{rccc}
\sharp: & V ^{\ast}\times V  \times \mathsf{F}\simeq T^{\ast}_{(\rho, \pi, \rho_{b})}(T^{\ast}V \times \mathsf{F}^{\ast}) & \to & V \times V^{\ast} \times \mathsf{F}^{\ast} \simeq T_{(\rho, \pi, \rho_{b})}(T^{\ast}V \times \mathsf{F}^{\ast})\\
& (e_{\rho},e_{\pi},e_{b}) & \mapsto & (e_{\pi}, -(-1)^{k(m-k)}e_{\rho}, -\iota^*_{\partial M}\,e_{\pi}).
\end{array}
\end{equation*}

This illustrates how the basic construction of the Dirac structure through duality pairing differs from that of our setting.

A main characteristic of our approach
is that it keeps intact all the fundamental structures of classical mechanics when passing from finite to infinite dimensional systems, including the variational principle, the canonical symplectic form, and the definition of the canonical Dirac structure. This is clearly seen by comparing Theorem \ref{theorem:ForcedLDkforms} and Theorem \ref{Theorem:LDfinitedim}, see also Table \ref{comparison}.
The treatment of boundary ports in our setting is automatically included through the definition of the restricted dual $V^\star$. The present framework also naturally unifies the treatment of systems with energy boundary flow cases on the space of functions, $V=C^{\infty}(M)$, and the space of $k$-forms, $V=\Omega^{k}(M)$, with $M$ of arbitrary dimension.
\end{remark}

\subsection{Example: Electromagnetism and the Poynting theorem}\label{sec:em}

In this example, the theory developed above is applied to electromagnetism on a compact manifold with boundary. Let $(M,g)$ be a compact Riemannian manifold with smooth boundary and $V=\Omega^1(M)$. The Lagrangian density reads
\begin{equation}\label{eq:lagrangiandensityem}
\mathscr L(A,\dot A,{\rm d}A)=\frac{1}{2} \dot A\wedge\star\dot A-\frac{1}{2}{\rm d}A\wedge\star{\rm d}A,\qquad(A,\dot A)\in T\Omega^1(M),
\end{equation}
where $\star:\Omega^k(M)\to\Omega^{m-k}(M)$ denotes the Hodge star operator, which is defined by the condition $g( \alpha , \beta ) \mu_g = \alpha\wedge\star\beta$ for each $\alpha,\beta\in\Omega^k(M)$, with $\mu_g\in\Omega^m(M)$ being the Riemannian volume form. Recall that $E=-\dot A\in\Omega^1(M)$ is the \emph{electric field} and $B={\rm d}A\in\Omega^2(M)$ is the \emph{magnetic field}. Note that we are using the Weyl gauge, although the treatment of other gauges is possible. A straightforward computations yields the following result.

\begin{lemma}\rm
If we write $\zeta={\rm d}A$, the partial derivatives of the Lagrangian density \eqref{eq:lagrangiandensityem} as defined in \eqref{PD_2} read
\[
\label{ParDer_Maxwell_2} \frac{\partial\mathscr L}{\partial A}=0,\qquad \displaystyle\frac{\partial\mathscr L}{\partial\dot A}=\star\dot A,\qquad  \displaystyle\frac{\partial\mathscr L}{\partial\zeta}=-\star\zeta.
\]
\end{lemma}

Let $\mathcal F :TV\to\Omega^{m-1}(M)$ and $\mathcal F _\partial:TV\to\Omega^{m-2}(\partial M)$ be the body and boundary external forces, respectively. From the previous Lemma and Proposition \ref{proposition:LDkforms}, the forced Lagrange--Dirac equations on $TV \oplus T^\star V$ read
\begin{equation*}
\left\{\begin{array}{ll}
\dot A=\nu, & \vspace{0.2cm}\\
\displaystyle\alpha=\star\nu,\qquad\quad & \displaystyle\dot\alpha=-{\rm d}\star{\rm d}A+\mathcal F ,\vspace{0.2cm}\\
\alpha_\partial=0, & \dot\alpha_\partial=-\iota_{\partial M}^*\left(\star{\rm d}A\right)+\mathcal F _\partial.
\end{array}\right.
\end{equation*}
By eliminating the variables $ \nu , \alpha , \alpha _ \partial $ and using the definition of the electric and magnetic fields, we get the system
\begin{equation}\label{Mawell_eq}
\left\{
\begin{array}{l}
\displaystyle\vspace{0.2cm}\star\dot E-{\rm d} \star B=-\mathcal{F},\\
\displaystyle \iota_{ \partial M} ^*\left(\star B\right)=\mathcal{F} _ \partial.
\end{array}
\right.
\end{equation}
In addition, we define the \emph{charge density} as $\rho=\star{\rm d}\star E\in\Omega^0(M)=C^\infty(M)$. From the first equation of \eqref{Mawell_eq}, we get
\begin{equation}\label{eq:chargeconservation}
\dot\rho+\star{\rm d}\mathcal F=0,
\end{equation}
where we have used that ${\rm d}\circ{\rm d}=0$.
When \eqref{Mawell_eq} and $\rho=\star{\rm d}\star E$ are augmented with the relations $ \dot  B= - {\rm d} E$ and $ {\rm d} B=0$, which follow from $E=- \dot  A$ and $B= {\rm d} A$, an equivalent writing of the Maxwell equations is obtained.

In order to interpret $ \mathcal{F}$ and $ \mathcal{F}_ \partial $, we compute the local energy balance. The energy density for electromagnetism is found as
\[
\mathscr{E}=\dot  A\wedge\frac{\partial \mathscr{L}}{\partial \dot  A} - \mathscr{L}= \dot  A\wedge\star \dot  A -\frac{1}{2}(\dot A\wedge\star\dot A-{\rm d}A\wedge\star{\rm d}A)
= \frac{1}{2}\left(E\wedge\star E +B\wedge\star B\right).
\]
As a result, the local and global energy balance equations computed in Proposition \ref{prop:energykforms} reads
\begin{equation}\label{Energy_Bal_Maxwell}
\frac{\partial \mathscr{E}}{\partial t}+{\rm d} (E\wedge\star B) =-E\wedge\mathcal F\quad\text{and}\quad\frac{d}{dt} \int_M \mathscr{E}=- \int_M E\wedge\mathcal{F} - \int_{ \partial M} \iota^*_{ \partial M} E\wedge\mathcal{F}_ \partial,
\end{equation}
which shows that $\mathcal{F} \in \Omega^{m-1}(M)$ and $\mathcal{F}_\partial \in \Omega^{m-2}(\partial M)$ describe the effects of the current density in the interior and the boundary current density, respectively.

To relate $\mathcal{F}$ and $\mathcal{F}_ \partial $ with the standard vector notations for currents, let $M\subset\mathbb R^3$ be a bounded domain with the Euclidean product, and denote by $\sharp:T^*M\to TM$ the sharp isomorphism defined by the Riemannian metric. Recall that there is an identification
\begin{equation*}
\Omega^1(M) ~\leftrightarrow~\mathfrak X(M),\quad a\mapsto\mathbf a=a^\sharp,
\end{equation*}
which also works for two-forms
\begin{equation*}
\Omega^2(M) ~\leftrightarrow~\mathfrak X(M),\quad b\mapsto\mathbf b=(\star b)^\sharp.
\end{equation*}

\begin{remark}[Riemannian metric on the boundary]\rm
Similar identifications hold on the boundary $\partial M$ by using the induced Riemannian metric, $g_\partial=\iota_{\partial M}^*g$. The sharp isomorphism and the Hodge star operator on $\partial M$ are denoted by $\sharp_\partial:T^*\partial M\to T\partial M$ and $\star_\partial:\Omega^k(\partial M)\to\Omega^{m-k-1}(\partial M)$, respectively.
Similarly, $\mu_g^\partial=\iota_{\partial M}^*(i_n\mu_g)\in\Omega^{m-1}( \partial M)$ is the Riemannian volume form on $\partial M$ (cf. \cite[Corollary 15.34]{Le2012}), where $n\in\mathfrak X(M)|_{\partial M}$ is the outward pointing, unit, normal vector field on $\partial M$, and $i_n:\Omega^k(M)\to\Omega^{k-1}(M)$ denotes the left interior multiplication by $n$. 
\end{remark}

Under these isomorphisms, we regard the electric and the magnetic fields as vector fields, $\mathbf E =E^\sharp$ and $\mathbf B =(\star B)^\sharp$, with analogous expressions for the body and boundary forces: $\mathbf J=(\star\mathcal F)^\sharp$ and $\mathbf j =(\star_\partial\mathcal F_\partial)^{\sharp_\partial}$, where $\mathbf J:T\Omega^1(M)\to \mathfrak X(M)$ is the \emph{current density} in the interior of $M$ and $\mathbf j:T\Omega^1(M)\to\mathfrak X(\partial M)$ is the \emph{surface current density} on the boundary. By using that $\star\star=(-1)^{k(m-k)}$ (here $k=1$), it is easy to check that the interior equation, i.e., the first equation in \eqref{Mawell_eq}, yields the \emph{Ampère law},
\begin{equation*}
\curl\mathbf B=\dot{\mathbf E}+\mathbf J.
\end{equation*}
The boundary condition in \eqref{Mawell_eq} yields the condition of adjacency to a perfect conductor (i.e., the magnetic field vanishes outside $M$),
\begin{equation*}
n\times \mathbf B|_{\partial M}=-\mathbf j,
\end{equation*}
where $n\in \mathfrak X(M)|_{\partial M}$ is the outward pointing, unit, normal vector field on the boundary and $\times$ denotes the cross product. In general, this boundary condition implies that the surface current density determines the jump in the tangential components of the magnetic field, with the external magnetic field being zero in this case.

Similarly, the charge density is expressed as $\rho={\rm div}\mathbf E$ and \eqref{eq:chargeconservation} is nothing but the \emph{charge conservation},
\begin{equation*}
\partial_t\rho+{\rm div}\mathbf J=0.
\end{equation*}

Furthermore, using again the above identification, the energy density $\mathscr E$ for electromagnetism yields the expression
\begin{equation*}
u=\star\mathscr E=\frac{1}{2}\left(|\mathbf E|^2+|\mathbf B|^2\right).
\end{equation*}
As a result, the local and global energy balance \eqref{Energy_Bal_Maxwell} becomes
the well-known \emph{Poynting theorem} and its global version, i.e.,
\begin{equation*}
\frac{\partial u}{\partial t}=-\operatorname{div}\mathbf S-\mathbf J\cdot\mathbf E \qquad \text{and}\qquad \frac{d}{dt} \int_M u \, {\rm d} x = -\int_M \mathbf{J} \cdot \mathbf{E} \,{\rm d} x + \int_ { \partial M} \mathbf{j} \cdot \mathbf{E} \,{\rm d} s,
\end{equation*}
where $\mathbf{S}=\mathbf{E}\times\mathbf{B}=(\star(E\wedge\star B))^\sharp$ is the \emph{Poynting vector}. Suitable extensions of the boundary conditions obtained above can also be derived from our approach, which will be analyzed in future work.

\begin{remark}[Constant external currents]\label{remark:r1}\rm
Following Remark \ref{remark:constantforces}, one can deal with constant external currents by adding an extra term on the Lagrangian. From this viewpoint, we have
\begin{equation*}
L_F(A,\dot A)=\int_M\frac{1}{2}\left(\dot A\wedge\star\dot A-{\rm d}A\wedge\star{\rm d}A\right)+\int_M A\wedge\mathcal J+\int_{\partial M}\iota_{\partial M}^* A\wedge\mathcal J_\partial,
\end{equation*}
where $\mathcal J\in\Omega^{m-1}(M)$ is the \emph{external current} and $\mathcal J_\partial\in\Omega^{m-2}(\partial M)$ is the \emph{external surface current}, which agree with the ones introduced before: $\mathbf J=(\star\mathcal J)^\sharp$ and $\mathbf j=(\star_\partial\mathcal J_\partial)^{\sharp_\partial}$.
\end{remark}

\section{Conclusion}

This paper has developed the foundations of a new geometric framework based on Lagrangian mechanics, variational principles, and infinite-dimensional Dirac structures to describe systems with boundary energy flow. A key feature of the proposed approach is that it satisfies a set of consistency criteria with the geometric framework of finite-dimensional mechanics. This is further illustrated in Table \ref{comparison} which compares finite and infinite dimensional Lagrange--Dirac systems.
A crucial step in the approach is the careful construction of the dual spaces, from which the expressions of the canonical symplectic form and canonical Dirac structures for systems with boundary energy flow were deduced. 
The applications to various examples, such as nonlinear wave equations, the telegraph equation, and the Maxwell equations, demonstrate the applicability and versatility of this approach.
In part II of this paper, systems described by bundle-valued $k$-forms will be considered, with application to gauge and particle field theories. Future work will focus on further developing the interconnection of systems within this framework, exploring applications to fluid dynamics and continuum mechanics, and investigating the role of symmetries.

\section*{Declarations}
The authors have no competing interests to declare that are relevant to the content of this article.

\section*{Acknowledgements}
F.G.-B. is partially supported by a start-up grant from the Nanyang Technological University. H.Y. is partially supported by JST CREST (JPMJCR24Q5), JSPS Grant-in-Aid for Scientific Research (22K03443) and Waseda University Grants for Special Research Projects (2025C-095).

\bibliographystyle{plainnat}
\bibliography{biblio}

@book{ChMa1974,
  title={Properties of Infinite Dimensional Hamiltonian Systems},
  author={Chernoff, P. and J.~E. Marsden},
  isbn={978-3-540-07011-5},
  year={1974},
    series={Lecture Notes in Mathematics},
    volume={425},
  publisher={Springer Berlin, Heidelberg}
}

@article{SeVanSch2012,
	author = {M. Seslija and  A. J. van der Schaft and J. M. A. Scherpen},
	journal = {Proc. 4th IFAC Workshop on {L}agrangian and {H}amiltonian Methods for Non Linear Control, IFAC},
	pages = {114-119},
	title = {Reduction of {S}tokes--{D}irac structures and gauge symmetry in port-{H}amiltonian systems},
	year = {2012}}

@article{VaYoLe2012,
	author = {J. Vankerschaver and  H. Yoshimura and M. Leok},
	journal = {J. Math. Phys.},
	title = {The {H}amilton-{P}ontryagin principle and multi-{D}irac structures for classical field theories},
	volume = {53},
	year = {2012}}

@article{GBYo2015,
	author = {F. Gay-Balmaz and H. Yoshimura},
	issn = {0196-8858},
	journal = {Adv. Appl. Math.},
	pages = {131-213},
	title = {{D}irac reduction for nonholonomic mechanical systems and semidirect products},
	volume = {63},
	year = {2015}}

@article{GBYo2020,
	author = {Gay-Balmaz, F. and Yoshimura, H.},
	doi = {10.1093/imamci/dnaa015},
	journal = {IMA J. Math. Control. Inf.},
	title = {{D}irac structures and variational formulation of port-{D}irac systems in nonequilibrium thermodynamics},
	volume = {37},
	year = {2020}}

@article{GBYo2018,
	author = {Gay-Balmaz, F. and Yoshimura, H.},
	journal = {J.Math. Phys.},
	title = {{D}irac structures in nonequilibrium thermodynamics},
	volume = {59},
	year = {2018}}

@article{GBYo2023,
	author = {F. Gay-Balmaz and H. Yoshimura},
	journal = {Phil. Trans. R. Soc. A.},
	title = {{S}ystems, variational principles and interconnections in non-equilibrium thermodynamics},
    volume = {381},
    number = {20220280},
	year = {2023}}

@article{YoMa2006a,
	Author = {H. Yoshimura and {J. E.} Marsden},
	fJournal = {Journal of Geometry and Physics},
	Journal = {J. Geom. Phys.},
	Number = {1},
	Pages = {133--156},
	Title = {{D}irac structures in {L}agrangian mechanics {P}art {I}: Implicit {L}agrangian systems},
	Volume = {57},
	Year = {2006}}

@article{YoMa2006b,
	Author = {H. Yoshimura and {J. E.} Marsden},
	fJournal = {Journal of Geometry and Physics},
	Journal = {J. Geom. Phys.},
	Number = {1},
	Pages = {209--250},
	Title = {{D}irac structures in {L}agrangian mechanics {P}art {II}: Variational structures},
	Volume = {57},
	Year = {2006}}

@article{GaMaRa2012,
author = {Gay-Balmaz, F. and Marsden, J. E.  and Ratiu, T. S.},
year = {2012},
month = {08},
pages = {},
title = {Reduced Variational Formulations in Free Boundary Continuum Mechanics},
volume = {22},
journal = {Journal of Nonlinear Science},
doi = {10.1007/s00332-012-9143-4}
}

@book{Le2012,
  author    = {J. Lee}, 
  title     = {Introduction to Smooth Manifolds},
  publisher = {Springer-Verlag New York},
  year      = 2012,
  edition   = {2nd},
  isbn      = {978-1-4419-9981-8}
}

@article{Co1990,
author = {T. J. Courant},
journal = {Transactions of the American Mathematical Society},
number = {2},
pages = {631--661},
publisher = {American Mathematical Society},
title = {{D}irac Manifolds},
volume = {319},
year = {1990}
}

@article{JaYo2014,
title = {Tensor products of {D}irac structures and interconnection in {L}agrangian mechanics},
journal = {J. Geom. Mech.},
volume = {6},
number = {1},
pages = {67-98},
year = {2014},
author = {H. O. Jacobs and H. Yoshimura},
}

@book{Ru1991,
  title={Functional Analysis},
  author={Rudin, W.},
  isbn={9780070542365},
  lccn={lc90005677},
  series={International series in pure and applied mathematics},
  year={1991},
  publisher={McGraw-Hill}
}

@book{MeMeVoRa1997,
  title={Introduction to Functional Analysis},
  author={Meise, R. and Vogt, D.},
  isbn={9780198514855},
  lccn={97012597},
  series={Oxford graduate texts in mathematics},
  year={1997},
  publisher={Clarendon Press}
}

@article{YoMa2007,
title = {Reduction of {D}irac structures and the {H}amilton-{P}ontryagin principle},
journal = {Reports on Mathematical Physics},
volume = {60},
number = {3},
pages = {381-426},
year = {2007},
issn = {0034-4877},
doi = {https://doi.org/10.1016/S0034-4877(08)00004-9},
url = {https://www.sciencedirect.com/science/article/pii/S0034487708000049},
author = {H. Yoshimura and J. E. Marsden}
}

@article{YoMa2009,
author = {Yoshimura, H. and Marsden, J. E.},
year = {2009},
pages = {},
title = {{D}irac cotangent bundle reduction},
volume = {1},
journal = {J. Geom. Mech.},
doi = {10.3934/jgm.2009.1.87}
}

@article{GBRAYo2025II,
author = {Gay-Balmaz, F. and Rodríguez Abella, {\'A}. and  Yoshimura, H.},
year = {2025},
pages = {},
title = {Infinite-dimensional {L}agrange--{D}irac systems with boundary energy flow {II}: Field theories with bundle-valued forms},
volume = {},
journal = {preprint}
}

@InProceedings{RAGBYo2023,
author="Rodr{\'i}guez Abella, {\'A}.
and Gay--Balmaz, F.
and Yoshimura, H.",
editor="Nielsen, Frank
and Barbaresco, Fr{\'e}d{\'e}ric",
title="Infinite Dimensional {L}agrange--{D}irac Mechanics with Boundary Conditions",
booktitle="Geometric Science of Information",
year="2023",
publisher="Springer Nature Switzerland",
address="Cham",
pages="202--211",
isbn="978-3-031-38299-4"
}

@article{GBYo2024,
title = {Variational and {D}irac structures for interconnected distributed-discrete systems},
journal = {IFAC-PapersOnLine},
volume = {58},
number = {6},
pages = {286-291},
year = {2024},
note = {8th IFAC Workshop on Lagrangian and Hamiltonian Methods for Nonlinear Control LHMNC 2024},
issn = {2405-8963},
author = {F. Gay-Balmaz and H. Yoshimura}
}

@article{MaVdS2005,
author = {Maschke, B. and van der Schaft, A. J.},
year = {2005},
title = {{C}ompositional modelling of distributed-parameter systems. In {F}. {L}amnabhi--{L}agarrigue, {A}. {L}oria, \& {E}. {P}anteley ({E}ds.), Advanced Topics in Control Systems Theory},
journal = {Advanced Topics in Control Systems Theory},
pages = {15-154},
volume = {311}
}

@article{VdSMa2002,
author = {van der Schaft, A. J. and Maschke, B.},
year = {2002},
pages = {166–194},
title = {{H}amiltonian formulation of distributed-parameter systems with boundary energy flow},
volume = {42},
journal = {J. Geom.
Phys.},
}

@article{HaMaSe2022,
author = {Haine, G. and Matignon, D. and Serhani, A.},
year = {2022},
title = {{N}umerical analysis of a structure-preserving space-discretization for an anisotropic and heterogeneous boundary controlled $N$-dimensional wave equation as port-Hamiltonian system},
journal = {arXiv:2006.15032},
}

@article{RaCaVdSSt2020,
author = {Rashad, R. and Califano, F. and van der Schaft, A. J. and Stramigioli, S.},
year = {2020},
title = {{T}wenty years of distributed port-{H}amiltonian systems: a literature review},
journal = {IMA Journal of Mathematical Control and Information},
volume = {37},
pages={1400–1422}
}

@inproceedings{VaYoLeMa2010,
  author={Vankerschaver, J. and Yoshimura, H. and Leok, M. and Marsden, J.E.},
  booktitle={49th IEEE Conference on Decision and Control (CDC)}, 
  title={{S}tokes--{D}irac structures through reduction of infinite-dimensional {D}irac structures}, 
  year={2010},
  volume={},
  number={},
  pages={6265-6270},
  doi={10.1109/CDC.2010.5717698}}

\end{document}